\documentclass[11pt]{amsart}
\usepackage[all]{xy}
\usepackage{amscd}
\usepackage{amsfonts}
\usepackage{amsmath}
\usepackage{amssymb}
\usepackage{amsthm}
\usepackage{bm}
\usepackage{CJK}
\usepackage{dsfont}
\usepackage{fancyhdr}
\usepackage{graphics}
\usepackage{graphicx}
\usepackage{latexsym}
\usepackage{lineno}
\usepackage{mathrsfs}
\usepackage{parskip}
\usepackage{slashed}
\usepackage{titletoc}
\usepackage[usenames]{color}
\allowdisplaybreaks

\begin{CJK*}{GBK}{kai}
\newtheorem{thm}{Theorem}[section]
\end{CJK*}

\newtheorem{lem}[thm]{Lemma}

\newtheorem{defn}[thm]{Definition}
\numberwithin{equation}{section}

\newcommand{\ric}{\mathrm{Ric}}
\newcommand{\tr}{\mathrm{tr}}
\newcommand{\ppt}{\frac{\partial}{\partial t}}

\newcommand{\mn}{\sqrt{-1}}

\newcommand{\ddbar}{\sqrt{-1}\partial\overline{\partial}}
\newcommand{\md}{\mathrm{d}}
\newcommand{\X}{ \mathfrak{X}(M)}
\newcommand{\N}{\mathcal{N}}
\newcommand{\cric}{\mathrm{Ric}^{(1,1)}}

\newcommand{\tmax}{T_{\mathrm{max}}}
\begin{document}
\begin{CJK}{GBK}{song}
\title{An almost complex Chern-Ricci flow}
\makeatletter
%\linenumbers 可以加行的号码, 但是不好用.
\let\uppercasenonmath\@gobble% disables title uppercase
\let\MakeUppercase\relax% disables author uppercase
\let\scshape\relax% disables section smallcaps
\makeatother
\author{Tao Zheng }
\address{Institut Fourier, Universit\'{e} Grenoble Alpes, 100 rue des maths,
Gi\`{e}res 38610, France}
\email{zhengtao08@amss.ac.cn}
%\date{ }
\thanks {Supported by National Natural Science Foundation of China grant No. 11401023, and the author's post-doc is supported by the European Research Council (ERC) grant No. 670846 (ALKAGE)}
\begin{abstract}
We consider the evolution of an almost Hermitian metric by the $(1,1)$ part of its Chern-Ricci form on almost complex manifolds. This is an evolution equation first studied by Chu and coincides with the Chern-Ricci flow if the complex structure is integrable and with the K\"ahler-Ricci flow if moreover the initial metric is K\"ahler. We find the maximal existence time for the flow in term of the initial data and also give a convergence result. As an example, we study this flow on the (locally) homogeneous manifolds in more detail.
\end{abstract}
\maketitle
\section{Introduction}
Let $(M,J)$ be a compact almost complex manifold (without boundary) with $\dim_{\mathbb{R}}M=2n$, where $J$ is the almost complex structure. Then assume that $g_0$ is an almost Hermitian metric on $(M,J)$, i.e., $g_0$ is a Riemannian metric satisfying $g_0(JX,JY)=g_0(X,Y)$ for any vector fields $X$ and $Y$. Associated to $g_0$ is a unique real $(1,1)$ form $\omega_0$ defined by $\omega_0(X,Y):=g_0(JX,Y)$ for any vector fields $X$ and $Y$ and vice versa. In what follows, we will not distinguish the two terms.

Since Hamilton \cite{Ha} introduced the Ricci flow, it has established many deep results in topological, smooth and Riemannian manifolds (see for example \cite{BS,Ha2,P1}). Next, we consider the parabolic flows of metrics on $M$  starting at $g_0$ which preserve the Hermitian condition and reveal the information about the structure of $M$ as a complex manifold. When $g_0$ is a K\"ahler metric, i.e., $\md\omega_0=0$, the Ricci flow does exactly this and it is called the K\"ahler-Ricci flow firstly introduced by Cao \cite{Ca}.
The behavior of the K\"ahler-Ricci flow is deeply intertwined with the complex and algebro-gemetric properties of $M$ (see for example \cite{Ca,CW,csw,FIK,PSSW,PS2,ST,ST2,ST3,SW,SW2,SW3,SW4,SY,szkrf,Ti,TZ,To1,tosattikawa,Ts,Ts2,weinkovekrf}).

If the Hermitian metric $g_0$ is not K\"ahler, then there are two types of Ricci curvature which are equal to each other when $\md\omega_0=0$. There are also two types of the evolution of the Hermitian metric. One is the Chern-Ricci flow which firstly introduced by Gill \cite{gill} when the first Bott-Chern class vanishes and is studied deeply by Tosatti and Weinkove (and Yang) \cite{twjdg,twcomplexsurface,twymathann}. The Chern-Ricci flow is a natural evolution equation on complex manifolds and its behavior reflects the underlying geometry (see also \cite{ftwz,gillmmp,gillscalar,laurent,lr,niexiaolan,yangxiaokui,zhengtaocjm} and references therein). Another type of the evolution of Hermitian metric was introduced by Streets and Tian \cite{StT,StT2,StT3} (see also \cite{LY}) and was generalized to almost Hermitian manifolds by Vezzoni \cite{V}.

Given $F\in C^{\infty}(M,\mathbb{R})$, Chu, Tosatti and Weinkove \cite{ctw} solved  the following Monge-Amp\`ere equation on almost Hermitian manifold
\begin{align}
\label{ecma}
&(\omega_0+\ddbar\varphi)^n= e^{F+b}\omega_0^n\\
&\omega_0+\ddbar\varphi>0, \quad\sup\varphi=0,\nonumber
\end{align}
for a unique $b\in\mathbb{R}$, where $\ddbar\varphi=\frac{1}{2}\left(\md J\md \varphi\right)^{(1,1)}$ (see  \eqref{ddbarvarphi} in Section \ref{sec2}).
When $J$ is integrable and $\omega_0$ is K\"ahler, \eqref{ecma} was solved by Yau \cite{Y} to confirm the Calabi conjecture. Tosatti and Weinkove \cite{TW2} solved \eqref{ecma} for any dimension if $J$ is integrable (see also \cite{Ch,TW1}).

In this paper, we consider the evolution equation for almost Hermitian forms suggested by Chu, Tosatti and Weinkove \cite{ctw}
\begin{align}
\label{acrf}
\frac{\partial\omega_t}{\partial t}=-\cric(\omega_t),\quad \omega(0)=\omega_0
\end{align}
on almost Hermitian manifold $(M,\omega_0,J)$. Here for convenience, we call $\cric(\omega)$ is the $(1,1)$ part of the Chern-Ricci form of $\omega$.
This flow coincides exactly with the Chern-Ricci flow if $J$ is integrable, and the K\"ahler-Ricci flow if furthermore $\omega_0$ is K\"ahler.

Chu \cite{chu1607} firstly studied the flow \eqref{acrf} in the case where there exists an almost Hermitian metric $\omega_0$ with $\cric(\omega_0)=\ddbar F$ for some $F\in C^{\infty}(M,\mathbb{R})$ and proved that the solution to \eqref{acrf} exists for all the time and converges to an almost Hermitian metric $\omega_{\infty}$ with $\cric(\omega_{\infty})=0$ (see more details in Section \ref{sec7}).

In this paper, we characterize the maximal existence time for a solution for the flow \eqref{acrf}. For this aim, we rewrite the flow as
\begin{align*}
\frac{\partial\omega_t}{\partial t}=-\cric(\omega_0)+\ddbar\theta(t),\quad\mbox{with}\quad\theta(t)=\log\frac{\omega_t^n}{\omega_0^n}.
\end{align*}
Hence, as long as the flow exists, the solution $\omega_t$ starting at $\omega_0$ must be of the form $\omega_t=\alpha_t+\ddbar\Upsilon(t)$ with $\frac{\partial \Upsilon}{\partial t}=\theta(t)$ and $\alpha_t=\omega_0-t\cric(\omega_0)$.
We define a number $T=T(\omega_0)$ with $0<T\leq \infty$ by
\begin{align}
T:=\sup\Big\{t\geq0:\;\exists\, \phi\in C^{\infty}(M,\mathbb{R})\,\mbox{such that}\,\alpha_t+\ddbar\phi>0\Big\}.
\end{align}
Note that for any other Hermitian metric $\omega_0'=\omega_0+\ddbar \psi>0$ with $\psi\in C^{\infty}(M,\mathbb{R})$, we have $T(\omega_0')=T(\omega_0)$.
 Indeed,  \eqref{diffric} yields that
\begin{align*}
\alpha_t+\ddbar\phi
=&\omega_0-t\cric(\omega_0')+\ddbar\left(\phi+t\log\frac{\omega_0^n}{\omega_0'^n}\right)\\
=&\omega_0'-t\cric(\omega_0')+\ddbar\left(\phi-\psi+t\log\frac{\omega_0^n}{\omega_0'^n}\right),
\end{align*}
as required.

It is easy to see that a solution to \eqref{acrf} cannot exist beyond time $T$. Indeed, we have
\begin{thm}
\label{mainthm}
There exists a unique maximal solution to the flow \eqref{acrf} on $[0,\,T)$.
\end{thm}
In the special case when $J$ is integrable, this is already known by the result of Tian and Zhang \cite{TZ} who extended earlier work of Cao \cite{Ca} and Tsuji \cite{Ts,Ts2} when $\omega_0$ is K\"{a}hler, and by the result of Tosatti and Weinkove \cite{twjdg} (see also \cite{twymathann}) when $\omega_0$ is even not K\"{a}hler.

We point out that the flow \eqref{acrf} is equivalent to the scalar partial differential equation
\begin{align}
\label{phit}
\ppt \phi(t)=\log \frac{(\alpha_t+\ddbar \phi_t)^n}{\omega_0^n}, \quad \alpha_t+\ddbar \phi_t>0,\quad\phi(0)=0,
\end{align}
on the same time interval.
Indeed, assume that $\phi_t$ is the solution to \eqref{phit}. We set $\omega_t=\alpha_t+\ddbar\phi_t>0$. From \eqref{diffric}, it follows that
\begin{align*}
\ppt \omega_t
=\ppt(\alpha_t+\ddbar \phi_t)
=-\cric(\omega_0)-\cric(\omega_t)+\cric(\omega_0)=-\cric(\omega_t),
\end{align*}
i.e., $\omega_t$ is the solution to the flow \eqref{acrf}.

On the other hand, assume that $\omega_t$ is the solution to \eqref{acrf}. We set
$$
\phi_t=\int_0^t\log \frac{\omega_s^n}{\omega_0^n}\md s
$$
with $\phi(0)=0$. This, together with \eqref{diffric}, yields
\begin{align*}
\ppt(\omega_t-\alpha_t-\ddbar\phi_t)
=-\cric(\omega_t)+\cric(\omega_0)+\cric(\omega_t)-\cric(\omega_0)=0,
\end{align*}
with $(\omega_t-\alpha_t-\ddbar\phi_t)|_{t=0}=0$, i.e., $\omega_t\equiv\alpha_t+\ddbar\phi_t$. It follows that $\phi_t$ is the solution to the scalar partial differential equation \eqref{phit}.

Next, we consider a convergence result  about the flow \eqref{acrf}. Since there is no Bott-Chern cohomology group on almost complex manifold if $J$ is not integrable, the statement  of this result may be slightly different.
\begin{thm}
\label{thm2}
Assume that $(M,\omega_0,J)$ is an almost Hermitian manifold  equipped a volume form $\Omega$  with $\cric(\Omega)<0$.  There exists an  almost Hermitian metric $\omega_{\infty}$ with $\cric(\omega_{\infty})=-\omega_{\infty}$, such that for any initial almost Hermitian metric $\omega_0$, the solution to \eqref{acrf} exists for all the time and that $\omega(t)/t$ converge smoothly to $\omega_{\infty}$ as $t\rightarrow\infty$.
\end{thm}
If the complex structure $J$ is integrable, then the existence of such volume form $\Omega$ is equivalent to the fact that $M$ is K\"ahler with negative first Chern class, and our theorem coincides precisely with \cite[Theorem 1.7]{twjdg}, i.e., on K\"ahler manifold with negative first Chern class, the Chern-Ricci flow, starting at any initial Hermitian metric $\omega_0$, exists for all the time and $\omega(t)/t$ converge smoothly to the unique K\"ahler-Einstein metric $\omega_{\mathrm{KE}}$ on $M$.

The outline of the paper is as follows. In Section \ref{sec2}, we collect some preliminaries about almost complex geometry. In Section \ref{sec3}, we give the uniform priori estimates of the solution to \eqref{apcma} and its time derivative. In Section \ref{sec4} and Section \ref{sec5}, we give the first and second order estimates of the solution to \eqref{apcma} respectively. In Section \ref{sec6}, we get the higher order estimates of the solution to \eqref{apcma} and complete the proof of Theorem \ref{mainthm}. In Section \ref{sec7}, we prove some convergence results including Theorem \ref{thm2}. In Section \ref{secexample}, as an example, we study the flow \eqref{acrf} on a compact almost Hermitian manifold $M$ whose universal cover is a Lie group $G$ such that if $\pi:G\longrightarrow M$ is the covering map, then $\pi^{\ast}\omega_0$ and $\pi^{\ast}J$ are left invariant in more detail.

\noindent {\bf Acknowledgements}
The author thanks Professor Valentino Tosatti for suggesting him this question and for many other invaluable conversations and directions.
The author also thanks Professor Jean-Pierre Demailly and Professor Ben Weinkove  for some useful comments. The author is also grateful to the anonymous referees and the editor for their careful reading and helpful suggestions which greatly improved the paper.
\section{Preliminaries}
\label{sec2}
In this section, we collect some basic materials about almost Hermitian geometry (see for example \cite{kobayashi,twyau}).
\subsection{The Nijenhuis tensor on almost complex manifolds}
Let $(M,J)$ be an almost complex manifold with $\dim_{\mathbb{R}}M=2n$, where $J$ is the almost complex structure, i.e., $J\in \mathrm{End}(TM)$ with $J^2=-\mathrm{Id}_{TM}$. Here $TM$ is the tangent vector bundle of $M$. Denote by $\X$ the set of all the sections of $TM$, i.e., the set of all the vector fields. Then we have
$$
TM\otimes_{\mathbb{R}}\mathbb{C}=T^{1,0}M\oplus T^{0,1}M,
$$
where
$$
T^{1,0}M:=\Big\{X-\mn JX:\;\forall\;X\in TM\Big\}
$$
and
$$
T^{0,1}M:=\Big\{X+\mn JX:\;\forall\; X\in TM\Big\}.
$$
Denote by $\Lambda^1 M$ the dual of $TM$.
We extend the Nijenhuis tensor $J$ to any $p$ form $\varpi$ by
\begin{align}
\label{jkuozhang}
(J\varpi)(X_1,\cdots,X_p):=(-1)^p\varpi(JX_1,\cdots,JX_p),\quad \forall\;X_1,\cdots,X_p\in \X.
\end{align}
It is easy to check that for any $(2n-p)$ form $\xi$ and $p$ form $\psi$, there holds
\begin{align}
\xi\wedge(J\psi)=(-1)^p(J\xi)\wedge\psi.
\end{align}
We also have
$$
\bigwedge^1 M\otimes_{\mathbb{R}}\mathbb{C}=\bigwedge^{1,0}M\oplus\bigwedge^{0,1}M,
$$
where
$$
\bigwedge^{1,0}M:=\left\{\eta+\mn J\eta:\;\forall\; \eta\in \bigwedge^1M\right\}
$$
and
$$
\bigwedge^{0,1}M:=\left\{\eta-\mn J\eta:\;\forall\; \eta\in \bigwedge^1M\right\}.
$$
It is easy to see that
$$
\left(T^{1,0}M\right)^{\ast}=\bigwedge^{1,0}M,\quad \left(T^{0,1}M\right)^{\ast}=\bigwedge^{0,1}M.
$$
We also denote that
$$
\bigwedge^{p,q}M:=\bigwedge^p\left(\bigwedge^{1,0}M\right)\otimes\bigwedge^q\left(\bigwedge^{0,1}M\right).
$$
Then we have
$$
\bigwedge^pM\otimes_{\mathbb{R}}\mathbb{C}=\oplus_{r+s=p}\bigwedge^{r,s}M.
$$
A $2$ form $\zeta$ is a $(1,1)$ form if and only if there holds
\begin{align}
\label{11biaozhun}
\zeta(X,Y)=\zeta(JX,JY),\quad \forall\,X,Y\in\X.
\end{align}
Indeed, $\zeta$ is a $(1,1)$ form if and only if
$$
\zeta(X+\mn JX,Y+\mn JY)=\zeta(X-\mn JX,Y-\mn JY)=0,\quad \forall\,X,Y\in\X,
$$
as required. Therefore, the $(1,1)$ part of any $2$ form $\xi$, denoted by $\xi^{(1,1)}$, can be given by
\begin{align}
\label{11part}
\xi^{(1,1)}(X,Y)=\frac{1}{2}\left(\xi(X,Y)+\xi(JX,JY)\right),\quad\forall\, X,Y\in\X.
\end{align}
We also define
\begin{align}
\xi^{\mathrm{ac}}(X,Y):=\xi(X,Y)-\xi^{(1,1)}(X,Y)=\frac{1}{2}\left(\xi(X,Y)-\xi(JX,JY)\right),\quad\forall\, X,Y\in\X.
\end{align}
\begin{defn}
For any $X,\,Y\in \X$,   Nijenhuis tensor $\N$  is defined by
\begin{align}
4\label{torsion}
\N(X,\,Y)=[JX,\,JY]-J[JX,\,Y]-J[X,\,JY]-[X,\,Y].
\end{align}
If $\N=0$, then the Nijenhuis tensor is called integrable.
\end{defn}
For the Nijenhuis tensor, a direct calculation yields
\begin{lem}\label{lemn}
The Nijenhuis tensor $\N$ satisfies
\begin{align*}
\N(X,\,Y)=&-\N( Y,\,X),\quad \N(JX,\,Y)=-J\N(X,\,Y),\\
\N(X,\,JY)=&-J\N(X,\,Y),\quad \N(JX,\,JY)=-\N(X,\,Y).
\end{align*}
Moreover, $\N(JX,\,Y)=\N(X,\,JY)$.
\end{lem}
By Lemma \ref{lemn}, for any   $(1,0)$ vector fields $W$ and $V$, it is easy to get
\begin{align*}
\N(V,\,W)=-[V,\,W]^{(0,1)},\quad \N(V,\,\overline{W})=\N(\overline{V},\,W)=0
\end{align*}
and
\begin{align*}
\N(\overline{V},\,\overline{W})=-[\overline{V},\,\overline{W}]^{(1,0)}=-\overline{[V,\,W]^{(0,1)}}.
\end{align*}
Let $e_1,\cdots, e_n$ be the basis of $T^{1,0}M$ and $\theta^1,\cdots,\theta^n$ be the dual basis of $\Lambda^{1,0}M$, i.e.,
$$
\theta^i(e_j)=\delta^i_j,\quad i,j=1,\cdots,n.
$$
Then we have
\begin{align}
\N(\overline{e}_i,\overline{e}_j)=&-[\overline{e}_i,\overline{e}_j]^{(1,0)}=:N_{\overline{i}\overline{j}}^ke_k,\\
\N(e_i,e_j)=&-[e_i,e_j]^{(0,1)}=\overline{N_{\overline{i}\overline{j}}^k}\overline{e}_k,
\end{align}
which implies
\begin{align*}
\N=\frac{1}{2}\overline{N_{\overline{i}\overline{j}}^k}\overline{e}_k\otimes\left(\theta^i\wedge\theta^j\right)+
\frac{1}{2}N_{\overline{i}\overline{j}}^ke_k\otimes\left(\overline{\theta}^i\wedge\overline{\theta}^j\right).
\end{align*}
Denote the structure coefficients of Lie bracket by
\begin{align*}
[e_i,e_j]=:&C_{ij}^ke_k-\overline{N_{\overline{i}\overline{j}}^k}\overline{e}_k,\\
[e_i,\overline{e}_j]=:&C_{i\overline{j}}^ke_k+C_{i\overline{j}}^{\overline{k}}\overline{e}_k,\\
[\overline{e}_i,\overline{e}_j]=:&-N_{\overline{i}\overline{j}}^ke_k+\overline{C_{ij}^k}\overline{e}_k.
\end{align*}
A direct calculation yields
\begin{align}
\label{dtheta}
\md\theta^i=-\frac{1}{2}C_{k\ell}^i\theta^k\wedge\theta^{\ell}
-C_{k\overline{\ell}}^i\theta^k\wedge\overline{\theta}^{\ell}
+\frac{1}{2}N_{\overline{k}\overline{\ell}}^i\overline{\theta}^k\wedge\overline{\theta}^{\ell}.
\end{align}
From \eqref{dtheta},  we can split the exterior differential operator, $\md: \Lambda^{\bullet}M\otimes_{\mathbb{R}}\mathbb{C}\longrightarrow \Lambda^{\bullet+1}M\otimes_{\mathbb{R}}\mathbb{C}$,  into four components (see for example \cite{angella})
$$
\md=A+\partial+\overline{\partial}+\overline{A}
$$
with
\begin{align*}
A:\;&\Lambda^{\bullet,\bullet} M\longrightarrow \Lambda^{\bullet+2,\bullet-1} M\\
\partial:\;&\Lambda^{\bullet,\bullet} M\longrightarrow \Lambda^{\bullet+1,\bullet} M\\
\overline{\partial}:\;&\Lambda^{\bullet,\bullet} M\longrightarrow \Lambda^{\bullet,\bullet+1} M\\
\overline{A}:\;&\Lambda^{\bullet,\bullet} M\longrightarrow \Lambda^{\bullet-1,\bullet+2} M.
\end{align*}
In terms of these components, the condition $\md^2=0$ can be written as
\begin{align*}
A^2=&0,\\
\partial A+A\partial=&0,\\
A\overline{\partial}+\partial^2+\overline{\partial}A=&0,\\
A\overline{A}+\partial\overline{\partial}+\overline{\partial}\partial+\overline{A}A=&0,\\
\partial\overline{A}+\overline{\partial}^2+\overline{A}\partial=&0,\\
\overline{A}\overline{\partial}+\overline{\partial}\overline{A}=&0,\\
\overline{A}^2=&0.
\end{align*}
For any $p$ form $\varpi$, there holds
\begin{align}
\label{waiweifen}
(\md\varpi)(X_1,\cdots,X_{p+1})
=&\sum\limits_{\lambda=1}^{p+1}(-1)^{\lambda+1}X_{\lambda}(\varpi(X_1,\cdots,\widehat{X_{\lambda}},\cdots,X_{p+1}))\\
&+\sum\limits_{\lambda<\mu}(-1)^{\lambda+\mu}\varpi([X_{\lambda},X_{\mu}],X_1,\cdots,\widehat{X_{\lambda}},\cdots,\widehat{X_{\mu}},\cdots,X_{p+1})\nonumber
\end{align}
for any fields $X_1,\cdots,X_{p+1}\in \X$.
For any $\varphi\in C^{\infty}(M,\,\mathbb{R})$, from \eqref{jkuozhang} and \eqref{waiweifen}, a direct computation yields
\begin{align}
\label{djdvarphi20}
(\md J\md\varphi)(e_i,e_j)
=&-2\mn[e_i,e_j]^{(0,1)}(\varphi),\\
\label{djdvarphi11}
(\md J\md\varphi)(\overline{e}_i,\overline{e}_j)
=&2\mn[\overline{e}_i,\overline{e}_j]^{(1,0)}(\varphi),\\
\label{djdvarphi02}
(\md J\md\varphi)(e_i,\overline{e}_j)
=&2\mn \left(e_i\overline{e}_j(\varphi)- [e_i,e_j]^{(0,1)}(\varphi)\right).
\end{align}
A direct calculation shows that
\begin{align}
\label{ddbarvarphi}
\ddbar \varphi=\frac{1}{2}(\md J\md\varphi)^{(1,1)} =\mn\left(e_i\overline{e}_j(\varphi)-[e_i,\overline{e}_j]^{(0,1)}(\varphi)\right)\theta^i\wedge\overline{\theta}^j.
\end{align}
Thanks to \eqref{torsion} and \eqref{waiweifen}, we get that
\begin{align*}
(\md J\md \varphi)(X,Y)-(\md J\md \varphi)(JX,JY)=-4\left(J\circ\N(X,Y)\right)(\varphi), \quad \forall\, X,\,Y\in \X,
\end{align*}
which, together with \eqref{11biaozhun}, yields that $\md J\md \varphi$ is a real $(1,1)$ form if and only if $J$ is integrable.

For any real $(1,1)$ form $\eta=\mn\eta_{i\overline{j}}\theta^i\wedge\overline{\theta}^j$, combining \eqref{jkuozhang} and \eqref{waiweifen} gives
\begin{align}
\label{dbareta}
\overline{\partial}\eta=\frac{\mn}{2}\left(\overline{e}_j\eta_{k\overline{i}}-\overline{e}_i\eta_{k\overline{j}}
-C_{k\overline{i}}^p\eta_{p\overline{j}}
+C_{k\overline{j}}^{p}\eta_{p\overline{i}}
+\overline{C_{ij}^q}\eta_{k\overline{q}}\right)\theta^k\wedge\overline{\theta}^i\wedge\overline{\theta}^j.
\end{align}
\subsection{The almost complex connection on almost complex manifolds} A connection $D$ is called almost complex if $DJ=0$, which means that for any $X,\,Y\in \X$, we have
\begin{align}\label{dj}
0=(D_XJ)(Y)=D_X(J(Y))-J(D_XY).
\end{align}
Therefore, we can define Christoffel symbols by
$$
D_{e_{\alpha}}e_j=\Gamma_{\alpha j}^ke_k,\quad D_{e_{\alpha}}\overline{e}_j=\Gamma_{\alpha\overline{j}}^{\overline{k}}\overline{e}_k,\;
\alpha  \in\{1,\cdots,n,\overline{1},\cdots,\overline{n}\},\quad i,\,j,\,k\in\{1,\cdots,n\}.
$$
Here we also use the notation $e_{\overline{i}}=\overline{e}_i$.
The connection form $(\omega_j^i)$ is defined by
$\omega_{j}^i:=\Gamma_{kj}^i\theta^k+\Gamma_{\overline{k}j}^i\overline{\theta}^k$. Then the torsion $\Theta=(\Theta^i)$ is defined by
$$
\Theta^i:=\md \theta^i-\theta^p\wedge\omega^i_p.
$$
This implies that
\begin{align}
\Theta^i(e_j,e_k)=&\Gamma_{jk}^i-\Gamma_{kj}^i-C_{jk}^i=:T_{jk}^i,\nonumber\\
\label{canonicald}\Theta^i(e_j,\overline{e}_k)=&-\Gamma_{\overline{k}j}^i-C_{j\overline{k}}^i=:S_{j\overline{k}}^i,\\
\Theta^i(\overline{e}_j,\overline{e}_k)=&-\md\theta^i([\overline{e}_j,\overline{e}_k])=N_{\overline{j}\overline{k}}^i.\nonumber
\end{align}
The torsion $\Theta=(\Theta^i)$ can be split into three parts
\begin{align*}
\Theta=\Theta^{2,0}+\Theta^{1,1}+\Theta^{0,2},
\end{align*}
where
\begin{align*}
\Theta^{2,0}=&\left(\frac{1}{2}T_{jk}^i\theta^j\wedge\theta^k\right)_{1\leq i\leq n},\\
\Theta^{1,1}=&\left(S_{j\overline{k}}^i\theta^j\wedge\overline{\theta}^k\right)_{1\leq i\leq n},\\
\Theta^{0,2}=&\left(\frac{1}{2}N_{\overline{j}\overline{k}}^i\overline{\theta}^j\wedge\overline{\theta}^k\right)_{1\leq i\leq n}.
\end{align*}
It follows that the $(0,2)$ part of an almost complex connection is uniquely determined by the Nijenhuis tensor (see for example \cite[Theorem 1.1]{kobayashi} and \cite{twyau}).

The curvature form is defined by
\begin{align*}
\Omega_{j}^{i}=\md\omega_{j}^{i}+\omega_{k}^{i}\wedge\omega_{j}^{k}.
\end{align*}
We define
\begin{align}
\label{rklji}
R_{k \ell j}{}^i
:=&\Omega_{j}^{i}(e_k,e_{\ell})\\
=&e_k(\Gamma_{\ell j}^i)-e_{\ell}(\Gamma_{k j}^i)-C_{k\ell}^p\Gamma_{pj}^i+\overline{N_{\overline{k}\overline{\ell}}^p}\Gamma_{\overline{p}j}^i
+\Gamma_{kp}^i\Gamma_{\ell j}^p-\Gamma_{\ell p}^i\Gamma_{kj}^p,\nonumber
\end{align}
\begin{align}
\label{riccikbarl}
R_{k \overline{\ell} j}{}^i
:=&\Omega_{j}^{i}(e_k,\overline{e}_{\ell})\\
=&e_k(\Gamma_{\overline{\ell} j}^i)-\overline{e}_{\ell}(\Gamma_{k j}^i)-C_{k\overline{\ell}}^p\Gamma_{pj}^i-C_{k\overline{\ell}}^{\overline{p}}\Gamma_{\overline{p}j}^i
+\Gamma_{kp}^i\Gamma_{\overline{\ell} j}^p-\Gamma_{\overline{\ell} p}^i\Gamma_{kj}^p,\nonumber
\end{align}
and
\begin{align}
\label{rkbarlbarji}
R_{\overline{k} \overline{\ell} j}{}^i
:=&\Omega_{j}^{i}(\overline{e}_k,\overline{e}_{\ell})\\
=&\overline{e}_k(\Gamma_{\overline{\ell} j}^i)-\overline{e}_{\ell}(\Gamma_{\overline{k} j}^i)+N_{\overline{k}\overline{\ell}}^p\Gamma_{pj}^i-\overline{C_{k\ell}^{p}}\Gamma_{\overline{p}j}^i
+\Gamma_{\overline{k}p}^i\Gamma_{\overline{\ell} j}^p-\Gamma_{\overline{\ell} p}^i\Gamma_{\overline{k}j}^p.\nonumber
\end{align}
We can write $\Omega=\left(\Omega_j^i\right)=\Omega^{(2,0)}+\Omega^{(1,1)}+\Omega^{(0,2)}$, with
\begin{align*}
\Omega^{(2,0)}=&\left(\frac{1}{2}R_{k\ell j}{}^i\theta^k\wedge\theta^{\ell}\right),\\
\Omega^{(1,1)}=&\left(R_{k\overline{\ell}j}{}^i\theta^k\wedge\overline{\theta}^{\ell}\right),\\
\Omega^{(0,2)}=&\left(\frac{1}{2}R_{\overline{k}\overline{\ell} j}{}^i\overline{\theta}^k\wedge\overline{\theta}^{\ell}\right).
\end{align*}
Then the Chern-Ricci form is $(\mn\Omega_{i}^i)\in 2\pi c_1(M,J)\in H^2(M,\mathbb{R})$ by the Chern-Weil theory (see for example \cite[Chapter 12]{kn2}),
where $c_1(M,J)$ is the first Chern class of $(M,J)$.
\subsection{The canonical connection on almost Hermitian manifolds}
Assume that $(M,\,g,\,J)$ is an almost Hermitian manifold with $\dim_{\mathbb{R}}M=2n$, where a Riemannian metric $g$ is called a Hermitian metric if it satisfies $g(X,\,Y)=g(JX,\,JY),\;\forall\; X,\,Y\in \X$. We also denote the Hermitian metric $g$  by $\langle\cdot,\,\cdot\rangle$. This metric uniquely defines a real $(1,1)$ form $\omega(\cdot,\,\cdot)=g(J\cdot,\, \cdot)$ and vice versa. An affine connection $D$ on $TM$ is called almost Hermitian connection if $D g=D J=0$.
For any $X,\,Y,\,Z\in \X$, we have
\begin{align}\label{dg}
0=(D_Xg)(Y,Z)=X\langle Y,\,Z\rangle-\langle D_XY,\,Z\rangle-\langle Y,\,D_XZ\rangle.
\end{align}
For the almost Hermitian connection, we have (see for example \cite[Theorem 2.1]{kobayashi} and \cite{gau})
\begin{lem}
Let $(M,\,g,\,J)$ be an almost Hermitian manifold with $\dim_{\mathbb{R}}M=2n$. Then for any given vector valued $(1,1)$ form $\Phi=(\Phi^i)_{1\leq i\leq n}$, there exists a unique almost Hermitian connection $D$ on $(M,\,g,\,J)$ such that the $(1,1)$ part if the torsion is equal to the given $\Phi$.
\end{lem}
If the $(1,1)$ part of the torsion of an almost Hermitian connection vanishes everywhere, then the connection is known as the \emph{second canonical connection} and was first introduced by Ehrensmann and Libermann \cite{ehrensmann}.
It is also sometimes referred to as the \emph{Chern connection}, since when $J$ is integrable it coincides with the connection defined in Chern \cite{chern}, and in Schouten and Danzig \cite{schouten}. We denote it by $\nabla^C$.

We use the local $(1,0)$ frames as before. We write $\omega$ as
$$
\omega=\mn g_{i\overline{j}}\theta^i\wedge\overline{\theta}^j.
$$
From \eqref{dj}, \eqref{canonicald} and \eqref{dg}, we get
\begin{align*}
e_i\langle e_k,\overline{e}_{\ell}\rangle
=\langle \nabla^C_{e_i}e_k,\overline{e}_{\ell}\rangle+\langle e_k,\nabla^C_{e_i}\overline{e}_{\ell}\rangle
=\langle \Gamma_{ik}^pe_p,\overline{e}_{\ell}\rangle+\langle e_k,\Gamma_{i\overline{\ell}}^{\overline{q}}\overline{e}_{q}\rangle
=\langle \Gamma_{ik}^pe_p,\overline{e}_{\ell}\rangle+\langle e_k,-\overline{C_{\ell\overline{i}}^{q}}\overline{e}_{q}\rangle,
\end{align*}
i.e.,
\begin{align}
\label{gammaikp}
\Gamma_{ik}^p=g^{\overline{\ell}p}e_ig_{k\overline{\ell}}+g^{\overline{\ell}p}g_{k\overline{q}}\overline{C_{\ell\overline{i}}^q}.
\end{align}
This gives the components of the torsion as
$$
T_{ik}^p=\Gamma_{ik}^p-\Gamma_{ki}^p-C_{ik}^p=g^{\overline{\ell}p}e_ig_{k\overline{\ell}}+g^{\overline{\ell}p}g_{k\overline{q}}\overline{C_{\ell\overline{i}}^q}
-g^{\overline{\ell}p}e_kg_{i\overline{\ell}}-g^{\overline{\ell}p}g_{i\overline{q}}\overline{C_{\ell\overline{k}}^q}-C_{i k}^p.
$$
We also lower the index of torsion and denote it by
$$
T_{ik\overline{\ell}}=T_{ik}^pg_{p\overline{\ell}}
=e_ig_{k\overline{\ell}}+g_{k\overline{q}}\overline{C_{\ell\overline{i}}^q}
-e_kg_{i\overline{\ell}}-g_{i\overline{q}}\overline{C_{\ell\overline{k}}^q}
-C_{ik}^pg_{p\overline{\ell}}.
$$
Thanks to \eqref{dbareta}, we obtain
\begin{align*}
\overline{\partial}\omega
= \frac{\mn}{2}\left(\overline{e}_jg_{k\overline{i}}-\overline{e}_ig_{k\overline{j}}
-C_{k\overline{i}}^pg_{p\overline{j}}
+C_{k\overline{j}}^{p}g_{p\overline{i}}
+\overline{C_{ij}^q}g_{k\overline{q}}\right)\theta^k\wedge\overline{\theta}^i\wedge\overline{\theta}^j
= \frac{\mn}{2}\overline{T_{ji\overline{k}}}\theta^k\wedge\overline{\theta}^i\wedge\overline{\theta}^j.
\end{align*}
By \eqref{gammaikp}, it yields that
\begin{align}
\label{gammaipp}
\Gamma_{ip}^p=e_i\left(\log\det g\right)-C_{i\overline{q}}^{\overline{q}}.
\end{align}
Using \eqref{rklji} and \eqref{gammaipp}, we can deduce that
\begin{align}
\label{rkl}
R_{k\ell}:=&R_{k \ell i}{}^i
=[e_k,e_{\ell}]^{(0,1)}\left(\log\det g\right)
-e_k\left(C_{\ell\overline{q}}^{\overline{q}}\right)
+e_{\ell}\left(C_{k\overline{q}}^{\overline{q}}\right)
+C_{k\ell}^pC_{p\overline{q}}^{\overline{q}}
+\overline{N_{\overline{k}\overline{\ell}}^p}C_{\overline{p}i}^i.
\end{align}
Combining \eqref{riccikbarl} and \eqref{gammaipp} implies
\begin{align}
\label{rkbarl}
R_{k\overline{\ell}}:=&R_{k \overline{\ell} i}{}^i
=-\left(e_k\overline{e}_{\ell}-[e_k,\overline{e}_{\ell}]^{(0,1)}\right)(\log\det g)
+\overline{e}_{\ell}\left(C_{i\overline{q}}^{\overline{q}}\right)
+e_k\left(C_{\overline{\ell} i}^i\right)
+C_{k\overline{\ell}}^pC_{p\overline{q}}^{\overline{q}}
-C_{k\overline{\ell}}^{\overline{p}}C_{\overline{p}i}^i.
\end{align}
From \eqref{rkbarlbarji} and \eqref{gammaipp}, it follows that
\begin{align}
\label{rkbarlbar}
R_{\overline{k} \overline{\ell}}:=&R_{\overline{k} \overline{\ell} i}{}^i
=-[\overline{e}_k,\overline{e}_{\ell}]^{(1,0)}\left(\log\det g\right)
-N_{\overline{k}\overline{\ell}}^pC_{p\overline{q}}^{\overline{q}}
+\overline{e}_k(C_{\overline{\ell} i}^i)
-\overline{e}_{\ell}(C_{\overline{k} i}^i)
-\overline{C_{k\ell}^{p}}C_{\overline{p}i}^i.
\end{align}
The Chern-Ricci form $\ric(\omega)$ is defined by
$$\ric(\omega):=\frac{\mn}{2}R_{k\ell}\theta^k\wedge\theta^{\ell}
+\mn R_{k\overline{\ell}}\theta^k\wedge\overline{\theta}^{\ell}
+\frac{\mn}{2}R_{\overline{k}\overline{\ell}}\overline{\theta}^k\wedge\overline{\theta}^{\ell}.
$$
It is a closed real $2$ form and furthermore is a closed real $(1,1)$ form if the complex structure is integrable. If $J$ is integrable and $\md\omega=0$, then the Chern-Ricci form coincides  exactly with the Ricci form defined by the Levi-Civita connection of $\omega$.
Assume that $\tilde \omega=\mn \tilde g_{i\overline{j}}\theta^i\wedge\overline{\theta}^j$ is another almost Hermitian metric.
From \eqref{djdvarphi20}, \eqref{djdvarphi11}, \eqref{djdvarphi02}, \eqref{rkl}, \eqref{rkbarl}
and \eqref{rkbarlbar}, it follows that (see also for example \cite[(3.16)]{twyau})
\begin{align}
\label{diffric}
\ric(\tilde\omega)-\ric(\omega)=-\frac{1}{2}\md J\md \log\frac{\tilde\omega^n}{\omega^n},
\end{align}
with $\ric(\omega)\in 2\pi c_1(M,J)\in H^2(M,\mathbb{R})$. Note that in general there exist representatives of $2\pi c_1(M,J)$ which cannot be written in the form $\ric(\omega)-\frac{1}{2}\md J\md F$ for any $\omega$ and $F$ even when $J$ is integrable (see for example \cite[Corollary 2]{TW1}).

The Chern scalar curvature $R$ is defined by
$$
R:=\tr_{\omega}\ric(\omega)=\tr_{\omega}\cric(\omega)=\frac{n\ric(\omega)\wedge\omega^{n-1}}{\omega^n}=\frac{n\cric(\omega)\wedge\omega^{n-1}}{\omega^n}.
$$
For any $\varphi\in C^{\infty}(M,\mathbb{R})$,  we define the canonical Laplacian by
\begin{align*}
\Delta^C_{\omega}\varphi:=\frac{n\ddbar\varphi\wedge\omega^{n-1}}{\omega^n}=\frac{n(\md J\md\varphi)\wedge\omega^{n-1}}{2\omega^n}
=g^{\overline{j}i}
\left(e_i\overline{e}_j(\varphi)-[e_i,\overline{e}_j]^{(0,1)}(\varphi)\right).
\end{align*}
Using the second canonical connection $\nabla^C$, it can also be rewritten as
$$\Delta_{\omega}^C\varphi=g^{\overline{j}i}\nabla^C_i\nabla^C_{\overline{j}}\varphi=g^{\overline{j}i}\nabla^C_{\overline{j}}\nabla^C_i\varphi$$
since the $(1,1)$ part of the torsion of $\nabla^C$ vanishes. Denote by $\Delta_g$ the Laplace-Beltrami operator of the Riemannian metric $g$. For these two different Laplace operators, we have (see for example \cite[Lemma 3.2]{To0})
\begin{align}
\label{2laplace}
\Delta_g \varphi=2\Delta_{\omega}^C\varphi+\tau(\md \varphi),
\end{align}
where
$$
\tau(\md \varphi)=2\mathrm{Re}\left(T_{pj}^jg^{\overline{q}p}\overline{e}_q(\varphi)\right).
$$
Given any volume form $\Omega$, there exists an almost Hermitian metric (not unique) $\omega'$ such that $\Omega=\omega'^n$ since we have $f:=\log\frac{\Omega}{\omega^n} \in C^{\infty}(M,\mathbb{R})$ and we can take $\omega'=e^{f/n}\omega$ for example. Hence, from \eqref{rkl}, \eqref{rkbarl}
and \eqref{rkbarlbar}, it follows that we can also define the Ricci form $\ric(\Omega)$ associated to $\Omega$ by replacing $\det g$ with $\Omega$ in \eqref{rkl}, \eqref{rkbarl}
and \eqref{rkbarlbar} when it occurs, and also have
\begin{align}
\label{omegaOmega}
\ric(\omega)-\ric(\Omega)=-\frac{1}{2}\md J\md\log\frac{\omega^n}{\Omega}.
\end{align}
\section{Preliminary estimates}
\label{sec3}
In this section, we give the estimates of $\varphi$ and $\dot\varphi:=\ppt\varphi$. For this aim, we need to prove that the flow \eqref{acrf} can be reduced to a parabolic Monge-Amp\`ere equation.
Fix $T_0<T$ and in particular $T_0<\infty$. By definition of $T$, we can define reference metrics $\hat\omega_t$ for $M\times[0,T_0]$ by
\begin{align*}
\hat\omega_t:=\alpha_t+\frac{t}{T_0}\ddbar\phi_{T_0}=\frac{T_0-t}{T_0}\omega_0+\frac{t}{T_0}\left(\alpha_{T_0}+\ddbar\phi_{T_0}\right)=:\mn\hat g_{i\overline{j}}\theta^i\wedge\overline{\theta}^j,
\end{align*}
where $\phi_{T_0}\in C^{\infty}(M,\mathbb{R})$  satisfies $\alpha_{T_0}+\ddbar\phi_{T_0}>0$. Note that the almost Hermitian metrics $\hat\omega_t$ vary smoothly on the compact interval $[0,T_0]$ and hence we can deduce uniform estimates on $\hat\omega_t$ for $t\in [0,T_0]$. We rewrite $\hat\omega_t$ as $\hat\omega_t=\omega_0+t\chi$ with
$$\chi=\frac{1}{T_0}\ddbar\phi_{T_0}-\cric(\omega_0).$$
We define a volume form $\Omega=\omega_0^ne^{\phi_{T_0}/T_0}$. Note that
$$\ddbar\log\Omega=\frac{1}{T_0}\ddbar\phi_{T_0}+\ddbar\log\omega_0^n\not=\chi=\frac{\partial\hat\omega_t}{\partial t}$$
and
\begin{align}
\label{hatomegaomega0}
C_0^{-1}\omega_0\leq \hat\omega_t\leq C_0\omega_0
\end{align}
for some uniform constant $C_0>0$.
\begin{lem}
\label{dengjiadingyi}
A smooth family $\omega(t)$ of almost Hermitian metrics on $[0,T_0)$ solves the flow \eqref{acrf} if and only if there is a family of smooth functions $\varphi(t)$ for $t\in [0,T_0)$ such that $\omega(t)=\hat\omega_t+\ddbar\varphi(t)$, and solve
\begin{align}
\label{apcma}
\frac{\partial}{\partial t}\varphi=\log\frac{\left(\hat\omega_t+\ddbar\varphi\right)^n}{\Omega},\quad \hat\omega_t+\ddbar\varphi>0,\quad \varphi|_{t=0}=0.
\end{align}
\end{lem}
\begin{proof}
We use the ideas from for example \cite{tosattikawa}.
For the ``if" direction, we set $\omega(t)=\hat\omega_t+\ddbar\varphi(t)$. From \eqref{diffric}, it follows that
\begin{align*}
\ppt\omega=
&\ppt\hat\omega_t+\ddbar\left(\ppt\varphi\right)\\
=&\frac{1}{T_0}\ddbar\phi_{T_0}-\cric(\omega_0)-\frac{1}{T_0}\ddbar\phi_{T_0}+\ddbar\log\frac{\omega^n}{\omega_0^n}\\
=&\frac{1}{T_0}\ddbar\phi_{T_0}-\cric(\omega_0)-\frac{1}{T_0}\ddbar\phi_{T_0}-\cric(\omega)+\cric(\omega_0),
\end{align*}
as required.

For the `only if' direction, assume that $\omega$ solves the flow \eqref{acrf} on  $[0,T_0)$. We define
$$
\varphi(t)=\int_{0}^t\log\frac{\omega(s)^n}{\Omega}\md s
$$
for $t\in [0,T_0)$. We have
\begin{align}
\label{gai1}
\ppt\varphi(t)=\log\frac{\omega(t)^n}{\Omega},\quad \varphi(0)=0.
\end{align}
On the other hand, by \eqref{diffric}, we can deduce
\begin{align}
\label{gai2}
\ppt\left(\omega-\hat\omega_t\right)
=-\cric(\omega)-\chi=\ddbar\left(\log\frac{\omega^n}{\omega_0^n}-\frac{\phi_{T_0}}{T_0}\right)=\ddbar\log\frac{\omega^n}{\Omega}.
\end{align}
Thanks to \eqref{gai1} and \eqref{gai2}, it follows that
$$
\ppt \left(\omega-\hat\omega_t-\ddbar\varphi\right)=0,\quad\mbox{with}\quad
\left(\omega-\hat\omega_t-\ddbar\varphi\right)|_{t=0}=0
$$
so that $\omega=\hat\omega_t+\ddbar\varphi$ and $\varphi$ is the solution to \eqref{apcma}.
\end{proof}

Standard parabolic theory of partial differential equation yields that there exists a unique maximal solution to \eqref{apcma} on some time interval $[0,\tmax)$ with $\tmax>0$. Assume for a contradiction that $\tmax<T_0$.

Now we prove uniform estimates for the solution $\varphi$ to \eqref{apcma} up to the maximal time.
For later use, we write
\begin{align*}
\omega(t)=\mn g_{i\overline{j}}\theta^i\wedge\overline{\theta}^j,\quad \omega_0=\mn (g_0)_{i\overline{j}}\theta^i\wedge\overline{\theta}^j.
\end{align*}
\begin{lem}
\label{lemc0}
Assume that $\varphi(t)$ is the solution to the flow \eqref{apcma} on $[0,\tmax)$.
There exists a positive uniform constant $C>0$, independent of $t\in [0,\tmax)$, such that
\begin{enumerate}
\item \label{c0}$\|\varphi(t)\|_{C^0}\leq C$.
\item \label{dotvarphi} $\|\dot{\varphi}(t)\|_{C^0}\leq C$.
\item \label{volumequi}$C^{-1}\omega_0^{n}\leq \omega^n\leq C\omega_0^n$.
\end{enumerate}
\end{lem}
\begin{proof}
We use the ideas from \cite{TZ,twjdg}.
For  Part \eqref{c0}, set $\psi=\varphi-At$ for a constant $A>0$ to be determined later. Fix any $T'\in (0,\tmax)$ and assume that $\psi$ attains a maximum
on $M\times [0,T']$ at a point $(x_0,t_0)$ with $t_0>0$. At this point,  the maximum principle yields
$$
0\leq \ppt \psi=\log\frac{\left(\hat\omega_t+\ddbar\psi\right)^n}{\Omega}-A\leq \log\frac{\hat\omega_t^n}{\Omega}-A<0,
$$
provided $A$ is chosen sufficiently large, a contradiction. Here we use the fact that $\hat\omega_t$ is a smooth family of metrics on $[0,\tmax]$. Since $T'\in (0,\tmax)$ is arbitrary, this yields that the maximum of $\psi(t)$ is achieved at $t=0$. We get the upper bound for $\psi$ and hence for $\varphi$. The lower  bound is proved similarly.

As for Part \eqref{dotvarphi}, we first consider the upper bound for $\dot{\varphi}$. We consider the quantity $Q_1=t\dot\varphi-\varphi-nt$ and have
\begin{align*}
\left(\ppt-\Delta_{\omega}^{C}\right)Q_1
=t\tr_{\omega}\chi-n+\tr_{\omega}(\omega-\hat\omega_t)=\tr_{\omega}(t\chi-\hat\omega_t)=-\tr_{\omega}\omega_0<0,
\end{align*}
where
we use the fact that $\ppt\dot\varphi=\Delta_{\omega}^C\dot\varphi+\tr_{\omega}\chi$.
This, together with the maximum principle, implies that upper bound for $Q_1$ and hence for $\dot\varphi$.

For the lower bound for $\dot\varphi$, we define
$Q_2:=(T_0-t)\dot\varphi+\varphi+nt$ and have
\begin{align*}
\left(\ppt-\Delta_{\omega}^{C}\right)Q_2
=(T_0-t)\tr_{\omega}\chi+n-\Delta_{\omega}^C\varphi=\tr_{\omega}\left(\hat\omega_t+(T_0-t)\chi\right)=\tr_{\omega}\hat\omega_{T_0}>0.
\end{align*}
Then the lower bound for $Q_2$ and hence for $\dot\varphi$ follows from the maximum principle and the fact that $\tmax<T_0$.

Part \eqref{volumequi} follows from  Part \eqref{dotvarphi}, \eqref{apcma} and the definition of $\Omega$.
\end{proof}
\section{First order estimate}
\label{sec4}
In this section, we give the first order estimate of the solution $\varphi$ to \eqref{apcma}.
\begin{thm}
\label{thm1order}
Assume that $\varphi(t)$ is the solution to the flow \eqref{apcma} on $[0,\tmax)$. There exists a uniform constant $C>1$, independent of $t\in [0,\tmax)$, such that
\begin{align}
\sup\limits_{M\times[0,\tmax)}|\partial \varphi|_{g_0} \leq C.
\end{align}
\end{thm}
\begin{proof}
We consider the quantity $Q:=e^{f(\varphi)}|\partial \varphi|_{g_0}^2$, where $f$ will be determined later.
Fix any $T'\in (0,\tmax)$ and assume that
$$
\sup\limits_{M\times[0,T']}Q=Q(x_0,t_0)
$$
with $t_0>0$. Since $g_0$ is almost Hermitian metric, we choose $g_0$-unitary frame $e_1,\cdots,e_n$ such that $g(x_0,t_0)$ is diagonal near $x_0$. At
the point $(x_0,t_0)$, the maximum principle yields
\begin{align}
\label{evolutionQ}
0\geq&\left(\Delta_{\omega}^C-\ppt\right)Q\\
=&e^f\left(\Delta_{\omega}^C-\ppt\right)|\partial \varphi|_{g_0}^2+|\partial \varphi|_{g_0}^2\left(\Delta_{\omega}^C-\ppt\right)e^f
 +2\mathrm{Re}\left(g^{\overline{i}i}e_i\left(|\partial \varphi|_{g_0}^2\right)\overline{e}_i(e^f)\right).\nonumber
\end{align}
Since $\omega=\hat\omega_t+\ddbar\varphi$, at $(x_0,t_0)$, a direct computation gives
\begin{align}
\label{lapef}
\Delta_{\omega}^C(e^f)
=&g^{\overline{i}i}\left(e_i\overline{e}_i\left(e^f\right)-[e_i,\overline{e}_i]^{(0,1)}\left(e^f\right)\right)\\
=&g^{\overline{i}i}\left(e^f(f')^2|e_i(\varphi)|^2+e^ff''|e_i(\varphi)|^2+e^ff'e_i\overline{e}_i(\varphi)
-e^ff'[e_i,\overline{e}_i]^{(0,1)}(\varphi)\right)\nonumber\\
=&g^{\overline{i}i}e^f\left((f')^2+f''\right)|e_i(\varphi)|^2+e^ff'\tr_{\omega}(\omega-\hat\omega_t)\nonumber\\
=&g^{\overline{i}i}e^f\left((f')^2+f''\right)|e_i(\varphi)|^2+ne^ff'-e^ff'\left(\tr_{\omega}\hat\omega_t\right)\nonumber
\end{align}
and
\begin{align}
\label{lappartialu}
\Delta^C_{\omega}(|\partial\varphi|_{g_0}^2)
=&g^{\overline{i}i}\left(e_i\overline{e}_i(e_k(\varphi)\overline{e}_k(\varphi))-[e_i,\overline{e}_i]^{(0,1)}(e_k(\varphi)\overline{e}_k(\varphi))\right)\\
=&g^{\overline{i}i}\left(|e_ie_k(\varphi)|^2+|e_i\overline{e}_k(\varphi)|^2+e_i\overline{e}_ie_k(\varphi)\overline{e}_k(\varphi)
+e_k(\varphi)e_i\overline{e}_i\overline{e}_k(\varphi)\right)\nonumber\\
&-g^{\overline{i}i}\left([e_i,\overline{e}_i]^{(0,1)}e_k(\varphi)\overline{e}_k(\varphi)+e_k(\varphi)
[e_i,\overline{e}_i]^{(0,1)}\overline{e}_k(\varphi)\right)\nonumber\\
=&g^{\overline{i}i}\left(|e_ie_k(\varphi)|^2+|e_i\overline{e}_k(\varphi)|^2\right)
+g^{\overline{i}i}\left(e_k\left(e_i\overline{e}_i(\varphi)-[e_i,\overline{e}_i]^{(0,1)}(\varphi)\right)\overline{e}_k(\varphi)\right)\nonumber\\
&+g^{\overline{i}i}e_k(\varphi)\left(\overline{e}_k\left(e_i\overline{e}_i(\varphi)-[e_i,\overline{e}_i]^{(0,1)}(\varphi)\right)\right)
-g^{\overline{i}i}[e_k,e_i]\overline{e}_i(\varphi)\overline{e}_k(\varphi)\nonumber\\
&-g^{\overline{i}i}e_i[e_k,\overline{e}_i](\varphi)\overline{e}_k(\varphi)
-g^{\overline{i}i}e_k(\varphi)[\overline{e}_k,e_i]\overline{e}_i(\varphi)
-g^{\overline{i}i}e_k(\varphi)e_i[\overline{e}_k,\overline{e}_i](\varphi)\nonumber\\
&-g^{\overline{i}i}\left[e_k,[e_i,\overline{e}_i]^{(0,1)}\right](\varphi)\overline{e}_k(\varphi)
-g^{\overline{i}i}e_k(\varphi)\left[\overline{e}_k,[e_i,\overline{e}_i]^{(0,1)}\right](\varphi)\nonumber\\
=&g^{\overline{i}i}\left(|e_ie_k(\varphi)|^2+|e_i\overline{e}_k(\varphi)|^2\right)
+g^{\overline{i}i}\left(e_k\left(g_{i\overline{i}}-\hat g_{i\overline{i}}\right)\overline{e}_k(\varphi)\right)\nonumber\\
&+g^{\overline{i}i}e_k(\varphi)\left(\overline{e}_k\left(g_{i\overline{i}}-\hat g_{i\overline{i}}\right)\right)
-g^{\overline{i}i}[e_k,e_i]\overline{e}_i(\varphi)\overline{e}_k(\varphi)\nonumber\\
&-g^{\overline{i}i}e_i[e_k,\overline{e}_i](\varphi)\overline{e}_k(\varphi)
-g^{\overline{i}i}e_k(\varphi)[\overline{e}_k,e_i]\overline{e}_i(\varphi)
-g^{\overline{i}i}e_k(\varphi)e_i[\overline{e}_k,\overline{e}_i](\varphi)\nonumber\\
&-g^{\overline{i}i}\left[e_k,[e_i,\overline{e}_i]^{(0,1)}\right](\varphi)\overline{e}_k(\varphi)
-g^{\overline{i}i}e_k(\varphi)\left[\overline{e}_k,[e_i,\overline{e}_i]^{(0,1)}\right](\varphi).\nonumber
\end{align}
On the other hand, we have, at $(x_0,t_0)$,
\begin{align}
\label{pptu}
\ppt|\partial\varphi|_{g_0}^2
=&e_k(\dot\varphi)\overline{e}_k(\varphi)+e_k(\varphi)\overline{e}_k(\dot\varphi)\\
=&g^{\overline{i}i}e_k(g_{i\overline{i}})\overline{e}_k(\varphi)
-e_k(\log\Omega)\overline{e}_k(\varphi)
+g^{\overline{i}i}e_k(\varphi)\overline{e}_k(g_{i\overline{i}})
-e_k(\varphi)\overline{e}_k(\log\Omega),\nonumber
\end{align}
where we use \eqref{apcma}. At $(x_0,t_0)$, from \eqref{lappartialu} and \eqref{pptu}, it follows that
\begin{align}
\label{evolutionpartialvarphi}
\left(\Delta^C_{\omega}-\ppt\right)|\partial\varphi|_{g_0}^2
=&g^{\overline{i}i}\left(|e_ie_k(\varphi)|^2+|e_i\overline{e}_k(\varphi)|^2\right)
-g^{\overline{i}i}e_k \left(\hat g_{i\overline{i}}\right)\overline{e}_k(\varphi)
-g^{\overline{i}i}e_k(\varphi)\overline{e}_k\left(\hat g_{i\overline{i}}\right) \\
&-g^{\overline{i}i}[e_k,e_i]\overline{e}_i(\varphi)\overline{e}_k(\varphi)
-g^{\overline{i}i}e_i[e_k,\overline{e}_i](\varphi)\overline{e}_k(\varphi)
-g^{\overline{i}i}e_k(\varphi)[\overline{e}_k,e_i]\overline{e}_i(\varphi)\nonumber\\
&-g^{\overline{i}i}e_k(\varphi)e_i[\overline{e}_k,\overline{e}_i](\varphi)
-g^{\overline{i}i}\left[e_k,[e_i,\overline{e}_i]^{(0,1)}\right](\varphi)\overline{e}_k(\varphi)\nonumber\\
&-g^{\overline{i}i}e_k(\varphi)\left[\overline{e}_k,[e_i,\overline{e}_i]^{(0,1)}\right](\varphi)
+2\mathrm{Re}\left(e_k(\log\Omega)\overline{e}_k(\varphi)\right)\nonumber\\
\geq&(1-\varepsilon)g^{\overline{i}i}\left(|e_ie_k(\varphi)|^2+|e_i\overline{e}_k(\varphi)|^2\right)\nonumber\\
&-\frac{C}{\varepsilon}|\partial \varphi|_{g_0}^2\sum\limits_{i}g^{\overline{i}i}
+2\mathrm{Re}\left(e_k(\log\Omega)\overline{e}_k(\varphi)\right),\nonumber
\end{align}
where for the inequality we use the Cauchy-Schwarz inequality and $\varepsilon \in(0,1/2]$ and without loss of generality we assume $|\partial \varphi|_{g_0}^2(x_0,t_0)>1$.
At $(x_0,t_0)$, a direct calculation implies
\begin{align}
\label{realpart}
&2\mathrm{Re}\left(g^{\overline{i}i}e_i\left(|\partial \varphi|_{g_0}^2\right)\overline{e}_i(e^f)\right)\\
=&2\mathrm{Re}\left(e^ff'g^{\overline{i}i}e_k(\varphi)e_i\overline{e}_k(\varphi)\overline{e}_i(\varphi)\right)
+2\mathrm{Re}\left(e^ff'g^{\overline{i}i}e_ie_k(\varphi)\overline{e}_k(\varphi)\overline{e}_i(\varphi)\right)\nonumber\\
=&2\mathrm{Re}\left(e^ff'g^{\overline{i}i}e_k(\varphi)\left(g_{i\overline{k}}-\hat g_{i\overline{k}}+[e_i,\overline{e}_k]^{(0,1)}(\varphi)\right) \overline{e}_i(\varphi)\right)\nonumber\\
&+2\mathrm{Re}\left(e^ff'g^{\overline{i}i}e_ie_k(\varphi)\overline{e}_k(\varphi)\overline{e}_i(\varphi)\right)\nonumber\\
\geq&2 e^ff'|\partial\varphi|_{g_0}^2
-2\mathrm{Re}\left(e^ff'g^{\overline{i}i}\hat g_{i\overline{k}}e_k(\varphi)\overline{e}_i(\varphi)\right)
-\varepsilon e^f(f')^2|\partial\varphi|_{g}^2|\partial\varphi|_{g_0}^2
-\frac{Ce^f}{\varepsilon}|\partial\varphi|_{g_0}^2\sum\limits_{i}g^{\overline{i}i}\nonumber\\
&-(1+2\varepsilon)e^f(f')^2|\partial\varphi|_{g_0}^2|\partial\varphi|_{g}^2-(1-\varepsilon)e^f\sum\limits_{i}g^{\overline{i}i}|e_ie_k(\varphi)|^2\nonumber\\
=&2 e^ff'|\partial\varphi|_{g_0}^2
-2\mathrm{Re}\left(e^ff'g^{\overline{i}i}\hat g_{i\overline{k}}e_k(\varphi)\overline{e}_i(\varphi)\right)
-\frac{Ce^f}{\varepsilon}|\partial\varphi|_{g_0}^2\sum\limits_{i}g^{\overline{i}i}\nonumber\\
&-(1+3\varepsilon)e^f(f')^2|\partial\varphi|_{g_0}^2|\partial\varphi|_{g}^2-(1-\varepsilon)e^f\sum\limits_{i}g^{\overline{i}i}|e_ie_k(\varphi)|^2,\nonumber
\end{align}
where for the inequality we use the Cauchy-Schwarz inequality and $\varepsilon\in(0,1/2]$. At $(x_0,t_0)$, from \eqref{apcma} and \eqref{lapef}, it follows that
\begin{align}
\label{evolutionef}
\left(\Delta_{\omega}^C-\ppt\right)e^f
=&e^f\left(f''+(f')^2\right)|\partial\varphi|_{g}^2+e^ff'\left(\Delta_{\omega}^C-\ppt\right)\varphi\\
=&e^f\left(f''+(f')^2\right)|\partial\varphi|_{g}^2+ne^ff'-e^ff'\left(\tr_{\omega}\hat\omega_t\right)-e^ff'\left(\ppt\varphi\right).\nonumber
\end{align}
From \eqref{evolutionQ}, \eqref{evolutionpartialvarphi}, \eqref{realpart} and \eqref{evolutionef}, it yields that, at $(x_0,t_0)$,
\begin{align}
\label{1oder1}
0\geq&e^f\left(f''-3\varepsilon (f')^2\right)|\partial\varphi|_{g_0}^2|\partial\varphi|_{g}^2 -e^ff'|\partial \varphi|_{g_0}^2\left(\tr_{\omega}\hat\omega_t\right)
-\frac{C_1e^f}{\varepsilon} |\partial \varphi|_{g_0}^2\sum\limits_{i}g^{\overline{i}i}\\
&-e^ff'\left(\ppt\varphi\right)|\partial\varphi|_{g_0}^2
-2e^f\mathrm{Re}\left(e_k(\log\Omega)\overline{e}_k(\varphi)\right)
+(n+2)e^ff'|\partial\varphi|_{g_0}^2\nonumber\\
&-2\mathrm{Re}\left(e^ff'g^{\overline{i}i}\hat g_{i\overline{k}}e_k(\varphi)\overline{e}_i(\varphi)\right),\nonumber
\end{align}
where $C_1>1$ is a uniform constant.

Define
$$
f(\varphi)=\frac{e^{-A\left(\varphi- \varphi_0-1\right)}}{A},\quad
\varepsilon=\frac{Ae^{A\left(\varphi(x_0,t_0)-\varphi_0-1\right)}}{6},
$$
where $\varphi_0:=\sup\limits_{M\times[0,\tmax)}\varphi$.
Then, at $(x_0,t_0)$, we have
\begin{align}
\label{1order2}
f''-3\varepsilon(f')^2=\frac{Ae^{-A\left(\varphi(x_0,t_0)-\varphi_0-1\right)}}{2},
\end{align}
\begin{align}
\label{1order3}
&-e^ff'|\partial \varphi|_{g_0}^2\left(\tr_{\omega}\hat\omega_t\right)
-\frac{C_1e^f}{\varepsilon} |\partial \varphi|_{g_0}^2\sum\limits_{i}g^{\overline{i}i}\\
\geq&e^f|\partial \varphi|_{g_0}^2\left(\sum\limits_{i}g^{\overline{i}i}\right)\left(C_{0}^{-1}-\frac{6C_1}{A}\right)
e^{-A\left(\varphi(x_0,t_0)-\varphi_0-1\right)},\nonumber
\end{align}
and
\begin{align}
\label{1order6}
2\left|\mathrm{Re}\left(e^ff'g^{\overline{i}i}\hat g_{i\overline{k}}e_k(\varphi)\overline{e}_i(\varphi)\right)\right|
\leq -2C_0e^ff'|\partial\varphi|_g^2
=2C_0e^f|\partial\varphi|_g^2e^{-A\left(\varphi(x_0,t_0)-\varphi_0-1\right)}.
\end{align}
where for the inequality we use \eqref{hatomegaomega0} and the fact that $f'<0$.

Note that we assume that $|\partial\varphi|_{g_0}(x_0,t_0)>1$. Choosing $A=12C_0C_1$, combining \eqref{1order2}, \eqref{1order3} and \eqref{1order6} yields that, at $(x_0,t_0)$,
\begin{align}
\label{1order4}
 e^f\left(f''-3\varepsilon (f')^2\right)|\partial\varphi|_{g_0}^2|\partial\varphi|_{g}^2-2\mathrm{Re}\left(e^ff'g^{\overline{i}i}\hat g_{i\overline{k}}e_k(\varphi)\overline{e}_i(\varphi)\right)
\geq  C^{-1}|\partial\varphi|_{g_0}^2|\partial\varphi|_{g}^2
\end{align}
and
\begin{align}
\label{1order5}
-e^ff'|\partial \varphi|_{g_0}^2\left(\tr_{\omega}\hat\omega_t\right)
-\frac{C_1e^f}{\varepsilon} |\partial \varphi|_{g_0}^2\sum\limits_{i}g^{\overline{i}i}
\geq C^{-1}|\partial \varphi|_{g_0}^2\sum\limits_{i}g^{\overline{i}i}.
\end{align}
where $C>0$ is a uniform constant. Combining \eqref{1oder1}, \eqref{1order4} and \eqref{1order5} implies that, at $(x_0,t_0)$,
\begin{align}
\label{1order7}
0\geq C^{-1}|\partial\varphi|_{g_0}^2|\partial\varphi|_{g}^2+C^{-1}|\partial \varphi|_{g_0}^2\sum\limits_{i}g^{\overline{i}i}-C|\partial\varphi|_{g_0}^2-C.
\end{align}
Note that here the constant $C$ may be larger.

From  Lemma \ref{lemc0} and the same argument in \cite{ctw} (see also \cite{chu1607}), Theorem \ref{thm1order} follows.
\end{proof}
Using Theorem \ref{thm1order}, taking $\varepsilon=1/2$ in \eqref{evolutionpartialvarphi}, we get
\begin{lem}
\label{lemevolutionpartialu}
There exists a uniform constant $C>0$ such that
\begin{align}
\label{lemformulaevolutionpartialvarphi}
\left(\Delta^C_{\omega}-\ppt\right)|\partial\varphi|_{g_0}^2
\geq\frac{1}{2}g^{\overline{i}i}\left(|e_ie_k(\varphi)|^2+|e_i\overline{e}_k(\varphi)|^2\right)
-C \sum\limits_{i}g^{\overline{i}i}
-C.
\end{align}
\end{lem}

\section{Second order estimate}
\label{sec5}
In this section, we deduce the second order estimate of the solution $\varphi$ to \eqref{apcma} using the method from \cite{ctw} in the elliptic setup.
\begin{thm}
\label{thm2order}
Assume that $\varphi$ is the solution to the flow \eqref{apcma}. There exists a uniform constant $C>0$, independent of $t$, such that
\begin{align}
\sup_{M\times[0,\tmax)}|\nabla^2 \varphi|_{g_0}\leq C,
\end{align}
where $\nabla$ is the Levi-Civita connection with respect to the Riemannian metric $g_0$.
\end{thm}
\begin{proof}
Denote by $\lambda_1(\nabla^2\varphi)\geq \cdots\geq \lambda_{2n}(\nabla^2\varphi)$ the eigenvalues of $\nabla^2\varphi$ with respect to the Riemannian metric $g_0$.
Then we have $\Delta_{g_0}\varphi=\sum\limits_{\alpha=1}\lambda_{\alpha}(\nabla^2\varphi)$, where $\Delta_{g_0}$ is the Laplace-Beltrami operator of $g_0$.
Since
$$
\omega=\hat\omega_t+\ddbar\varphi>0,
$$
it follows that
\begin{align*}
\Delta^C_{\omega_0}\varphi=\tr_{\omega_0}\omega-\tr_{\omega_0}\hat\omega_t>-\tr_{\omega_0}\hat\omega_t>-C,
\end{align*}
where we use the fact that $\hat\omega_t$ vary smoothly on the compact interval $[0,T_0]$ and hence can be estimated uniformly.
This, together with \eqref{2laplace} and Theorem \ref{thm1order}, yields
\begin{align}
\label{lapg0varphi}
\Delta_{g_0}\varphi
=&2\Delta_{\omega_0}^C\varphi+\tau(\md \varphi)>-C,
\end{align}
where $C>0$ is a uniform constant. As a result, we get
\begin{align}
|\nabla^2 \varphi|_{g_0}\leq C\max\Big\{\lambda_1(\nabla^2\varphi),0\Big\}+C.
\end{align}
Therefore, it is sufficient to bound $\lambda_1(\nabla^2\varphi)$ from above in order to prove Theorem \ref{thm2order}.
Denote
$$
S_H:=\Big\{(x,t)\in M\times [0,\tmax):\;\lambda_1(\nabla^2\varphi)>0\Big\}.
$$
Without loss of generality, we assume $S_H\not=\emptyset$. Otherwise, we can get the upper bound of $\lambda_1(\nabla^2\varphi)$ directly.
We consider the quantity
$$
H=\log\lambda_1(\nabla^2\varphi)+\phi(|\partial \varphi|_g^2)+h(\varphi-\varphi_0),
$$
on the set $S_H$,
where
$$
\phi(s)=-\frac{1}{2}\log\left(1-\frac{s}{2K}\right),\quad h(s)=e^{-Ds}
$$
with sufficiently large uniform constant $D$ to be determined later. Here recall that $\varphi_0:=\sup\limits_{M\times[0,\tmax)}\varphi$, and for convenience, we use notation $K:=1+\sum\limits_{M\times[0,\tmax)}|\partial \varphi|_{g_0}^2$ which is a uniform constant by Theorem \ref{thm1order}. Note that
$$
\phi(|\partial \varphi|_{g_0}^2)\in\,[0,\,2\log 2]
$$
and
$$
\frac{1}{4K}\leq \phi'\leq \frac{1}{2K},\quad \phi''=2(\phi')^2.
$$
For any $T'\in (0,\tmax)$, assume that $H$ attains its maximum at $(x_0,t_0)$ on $\Big\{(x,t)\in M\times [0,T']:\;\lambda_1(\nabla^2\varphi)>0\Big\}$.
Note that $H$ is not smooth in general since the dimension of eigenspace associated to $\lambda_1(\nabla^2\varphi)$ may be strictly larger than $1$. We use a perturbation argument in \cite{ctw} to deal with this case (see also \cite{sz,stw1503}). Since $g_0$ is almost Hermitian metric, we can choose a coordinate patch $(U;\,x^1,\cdots,x^{2n})$ centered at $x_{0}$ such that
\begin{enumerate}
\item $U$ is diffeomorphic to a ball $B_2(0)\subset \mathbb{R}^{2n}$ of radius $2$ centered at $0$.
\item Denote by $\partial_{\alpha}$ the local vector fields $\partial/\partial x^{\alpha}$. There holds
$g_0(\partial_{\alpha},\partial_{\beta})|_{x_{0}}=\delta_{\alpha\beta},\,\alpha,\,\beta\hat{}=1,\cdots,2n$.
\item The almost complex structure $J$ satisfies $J(x_{0})=J_0$, where $J_0$ is the standard complex structure on $\mathbb{R}^{2n}$, i.e., $J_0\partial_{2i-1}=\partial_{2i}$ for $i=1,\cdots,n$.
\item There holds
\begin{align}
\label{zuobiao}
\partial_{\gamma}(g_0)_{\alpha\beta}|_{x_{0}}=0,\quad \forall\;\alpha,\,\beta,\,\gamma=1,\cdots,2n.
\end{align}
\item If at $x_0$, we define
\begin{align}
\label{10fields}
e_i:=\frac{1}{\sqrt{2}}\left(\partial_{2i-1}-\mn\partial_{2i}\right),\quad i=1,\cdots,n,
\end{align}
then these form a frame of $(1,0)$ vectors at $x_0$, and we have $(g_0)_{i\overline{j}}:=g_0(e_i,\overline{e}_j)=\delta_{ij}$, i.e., the frame is $g_0$-unitary. Furthermore, the argument in \cite{ctw} yields that we can assume $\left(g_{i\overline{j}}(x_0,t_0)\right)$ is diagonal with
$$
g_{1\overline{1}}(x_0,t_0)\geq\cdots\geq g_{n\overline{n}}(x_0,t_0).
$$
\end{enumerate}
We extend $e_1,\cdots,e_n$ smoothly to a $g_0$-unitary frame of $(1,0)$ vectors in a neighborhood of $x_0$, and from now on, the local unitary frame is fixed. Assume that  $V_1\in T_{x_0}M$ is a unit vector, i.e., $g_0(V_1,V_1)=1$, with
$$
\nabla^2\varphi(V_1,V_1)=\lambda_1(\nabla^2\varphi)(x_0,t_0).
$$
We can construct an orthonormal basis of $T_{x_0}M$, denoted by $V_1,\cdots,V_{2n}$, such that
$$
\nabla^2\varphi(V_{\alpha},V_{\alpha})=\lambda_{\alpha}(\nabla^2\varphi)(x_0,t_0), \quad\alpha=1,\cdots,2n
$$
with $\lambda_{1}(\nabla^2\varphi)(x_0,t_0)\geq \cdots\geq \lambda_{2n}(\nabla^2\varphi)(x_0,t_0)$. We assume
$V_{\beta}=\sum\limits_{\alpha=1}^{2n}V^{\alpha}{}_{\beta}\partial_{\alpha}$ for $\beta=1,\cdots,2n$. We extend $V_1,\cdots,V_{2n}$ to be vector fields in a neighborhood of $x_0$ by taking the components to be constant.
To use the perturbation argument, we define a smooth section $B=B_{\alpha\beta}\md x^{\alpha}\otimes \md x^{\beta}$ of $T^{\ast}M\otimes T^{\ast}M$ near $x_0$, where
$$
B_{\alpha\beta}=\delta_{\alpha\beta}-V^{\alpha}{}_1V^{\beta}{}_1.
$$
It is easy to deduce that the eigenvalues of $(B_{\alpha\beta})$  are $0,\underbrace{1,\cdots,1}_{(n-1)}$ and that $V_1$ is the eigenvector associated to the eigenvalue $0$, i.e., $B(V_1,V_1)=0$. Consider the endomorphism $\Phi=\left(\Phi^{\alpha}{}_{\beta}\right)$ of $TM$ defined by
\begin{align}
\label{phi}
\Phi^{\alpha}{}_{\beta}=g_0^{\alpha\gamma}\nabla^2_{\gamma\beta}\varphi-g_0^{\alpha\gamma}B_{\gamma\beta},
\end{align}
and denote its eigenvalues by
$$
\lambda_1(\Phi)\geq\cdots\geq\lambda_{2n}(\Phi).
$$
Since $B=\left(B_{\alpha\beta}\right)$ is nonnegative, we can deduce that $\lambda_1(\Phi)\leq \lambda_1(\nabla^2\varphi)$ in the neighborhood of $(x_0,t_0)$.
In particular,  we have
$$
\lambda_{1}(\Phi)(x_0,t_0)=\lambda_{1}(\nabla^2\varphi)(x_0,t_0),\quad\Phi(V_{\alpha})(x_0,t_0)=\lambda_{\alpha}(\Phi)(x_0,t_0)V_{\alpha},\quad \alpha=1,\cdots,2n
$$
and  hence
$$
\lambda_1(\Phi)(x_0,t_0)>\lambda_2(\Phi)(x_0,t_0)\geq\cdots\geq\lambda_{2n}(\Phi)(x_0,t_0).
$$
 This yields that $\lambda_1(\Phi)$ is smooth in a neighborhood of $(x_0,t_0)$. In the following, we write $\lambda_{\alpha}$ for $\lambda_{\alpha}(\Phi)$
for short. We can apply the maximum principle to the quantity
$$
\tilde H=\log \lambda_1(\Phi)+\phi(|\partial \varphi|_{g_0}^2)+h(\varphi-\varphi_0),
$$
which still obtains a local maximum at $(x_0,t_0)$.

We need the first and second derivatives of $\lambda_1$ at $(x_0,t_0)$ as follows.
\begin{lem}[Chu, Tosatti and Weinkove \cite{ctw}; 2016]
At $(x_0,t_0)$, we have
\begin{align}
\label{lambda1order}
\lambda_1^{\alpha\beta}:=&\frac{\partial \lambda_1}{\partial\Phi^{\alpha}{}_{\beta}}=V^{\alpha}{}_1V^{\beta}{}_1.\\
\label{lambda2order}
\lambda_1^{\alpha\beta,\gamma\delta}:=&\frac{\partial^2\lambda_1}{\partial\Phi^{\alpha}{}_{\beta}\partial\Phi^{\gamma}{}_{\delta}}
=\sum\limits_{\mu>1}\frac{V^{\alpha}{}_1V^{\beta}{}_{\mu}V^{\gamma}{}_{\mu}V^{\delta}{}_1+
V^{\alpha}{}_{\mu}V^{\beta}{}_{1}V^{\gamma}{}_{1}V^{\delta}{}_{\mu}}{\lambda_1-\lambda_{\mu}},
\end{align}
where the Greek indices  $\alpha,\beta,\gamma,\mu,\cdots$ go from $1$ to $2n$, unless otherwise indicated.
\end{lem}
By Lemma \ref{lemc0} and \eqref{apcma}, the arithmetic-geometry mean inequality yields
\begin{align}
\label{tromegaomega0xiajie}
\tr_{\omega}\omega_0\geq c,
\end{align}
for a uniform constant $c>0$. We also assume that, at $(x_0,t_0)$, there holds $\lambda_1\gg K\geq1$.
\begin{lem}
\label{lem1}
At $(x_0,t_0)$, we have
\begin{align}
\label{lem1formula}
\left(\Delta_{\omega}^C-\ppt\right)\lambda_1
\geq&2\sum\limits_{\alpha>1}g^{\overline{i}i}\frac{|e_i(\varphi_{V_{\alpha}V_{1}})|^2}{\lambda_1-\lambda_{\alpha}}
+g^{\overline{p}p}g^{\overline{q}q}\left|V_1(g_{p\overline{q}})\right|^2 \\
&-2g^{\overline{i}i}[V_1,e_i]V_1\overline{e}_i(\varphi)
-2g^{\overline{i}i}[V_1,\overline{e}_i]V_1\overline{e}_i(\varphi)-C\lambda_1\sum\limits_{i}g^{\overline{i}i},\nonumber
\end{align}
where we write
$$
\varphi_{\alpha\beta}:=\nabla^2_{\alpha\beta}\varphi,\quad \varphi_{V_{\alpha}V_{\beta}}:=\varphi_{\gamma\delta}V^{\gamma}{}_{\alpha}V^{\delta}{}_{\beta}=\nabla^2\varphi(V_{\alpha},V_{\beta}).
$$
\end{lem}
\begin{proof}
At $(x_0,t_0)$, noting that $g(x_0,t_0)$ is diagonal, by \eqref{lambda1order} and \eqref{lambda2order}, a direct computation yields
\begin{align}
\label{Llambda1}
\left(\Delta_{\omega}^C-\ppt\right)\lambda_1
=&g^{\overline{i}i}\lambda_{1}^{\alpha\beta,\gamma\delta}e_i(\Phi^{\gamma}{}_{\delta})\overline{e}_i(\Phi^{\alpha}{}_{\beta})
+g^{\overline{i}i}\lambda_{1}^{\alpha\beta}e_i\overline{e}_i(\Phi^{\alpha}{}_{\beta})\\
&-g^{\overline{i}i}\lambda_{1}^{\alpha\beta}[e_i,\overline{e}_i]^{(0,1)}(\Phi^{\alpha}{}_{\beta})
-\lambda_{1}^{\alpha\beta}\ppt (\Phi^{\alpha}{}_{\beta})\nonumber\\
=&g^{\overline{i}i}\lambda_{1}^{\alpha\beta,\gamma\delta}e_i(\varphi_{\gamma\delta})\overline{e}_i(\varphi_{\alpha\beta})
+g^{\overline{i}i}\lambda_{1}^{\alpha\beta}e_i\overline{e}_i(\varphi_{\alpha\beta})\nonumber\\
&+g^{\overline{i}i}\lambda_{1}^{\alpha\beta}\varphi_{\gamma\beta}e_i\overline{e}_i(g_0^{\alpha\gamma})
-g^{\overline{i}i}\lambda_{1}^{\alpha\beta}B_{\gamma\beta}e_i\overline{e}_i(g_0^{\alpha\gamma})\nonumber\\
&-g^{\overline{i}i}\lambda_{1}^{\alpha\beta}[e_i,\overline{e}_i]^{(0,1)}(\varphi_{\alpha\beta})
-\lambda_{1}^{\alpha\beta}\nabla_{\alpha\beta}^2\dot{\varphi}\nonumber\\
\geq&2\sum\limits_{\alpha>1}g^{\overline{i}i}\frac{|e_i(\varphi_{V_{\alpha}V_{1}})|^2}{\lambda_1-\lambda_{\alpha}}
+g^{\overline{i}i}e_i\overline{e}_i(\varphi_{V_1V_1})
-g^{\overline{i}i}[e_i,\overline{e}_i]^{(0,1)}(\varphi_{V_1V_1})\nonumber\\
&-\dot{\varphi}_{V_1V_1}-C\lambda_1\sum\limits_{i}g^{\overline{i}i}.\nonumber
\end{align}
Since $\varphi_{V_{\alpha}V_{\beta}}=V_{\alpha}V_{\beta}(\varphi)-\left(\nabla_{V_{\alpha}}V_{\beta}\right)(\varphi)$ for any $\alpha$ and $\beta$, we have
\begin{align}
\label{eieiuv1v11}
g^{\overline{i}i}\left(e_i\overline{e}_i-[e_i,\overline{e}_i]^{(0,1)}\right)(\varphi_{V_1V_1})
=g^{\overline{i}i}\left(e_i\overline{e}_i-[e_i,\overline{e}_i]^{(0,1)}\right)\left( V_1V_1(\varphi)-\left(\nabla_{V_{1}}V_{1}\right)(\varphi)\right).
\end{align}
From now on, we write $G$ for term bounded by $C\lambda_1\sum\limits_{i}g^{\overline{i}i}$ which may change from line to line.
At $(x_0,t_0)$, noting that $|\partial \varphi|_{g_0}\leq C$, a direct computation yields
\begin{align}
\label{eieiuv1v12}
&g^{\overline{i}i}\left(e_i\overline{e}_i-[e_i,\overline{e}_i]^{(0,1)}\right)\left(\nabla_{V_{1}}V_{1}\right)(\varphi)\\
=&g^{\overline{i}i}e_i\left(\nabla_{V_{1}}V_{1}\right)\overline{e}_i(\varphi)
-g^{\overline{i}i}\left(\nabla_{V_{1}}V_{1}\right)[e_i,\overline{e}_i]^{(0,1)}(\varphi)+G\nonumber\\
=&g^{\overline{i}i}\left(\nabla_{V_{1}}V_{1}\right)e_i\overline{e}_i(\varphi)
-g^{\overline{i}i}\left(\nabla_{V_{1}}V_{1}\right)[e_i,\overline{e}_i]^{(0,1)}(\varphi)
+G\nonumber\\
=&g^{\overline{i}i}\left(\nabla_{V_{1}}V_{1}\right)\left(e_i\overline{e}_i(\varphi)-[e_i,\overline{e}_i]^{(0,1)}(\varphi)\right)
+G\nonumber\\
=&g^{\overline{i}i}\left(\nabla_{V_{1}}V_{1}\right)\left(g_{i\overline{i}}-\hat g_{i\overline{i}}\right)
+G\nonumber\\
=&g^{\overline{i}i}\left(\nabla_{V_{1}}V_{1}\right)(g_{i\overline{i}})+G\nonumber
\end{align}
and
\begin{align}
\label{giieieiv1v1}
&g^{\overline{i}i}\left(e_i\overline{e}_iV_{1}V_{1}(\varphi)-[e_i,\overline{e}_i]^{(0,1)}V_1V_1(\varphi)\right)\\
=&g^{\overline{i}i}V_1V_1\left(e_i\overline{e}_i(\varphi)-[e_i,\overline{e}_i]^{(0,1)}(\varphi)\right)
-2g^{\overline{i}i}[V_1,e_i]V_1\overline{e}_i(\varphi)-2g^{\overline{i}i}[V_1,\overline{e}_i]V_1e_i(\varphi)+G\nonumber\\
=&g^{\overline{i}i}V_1V_1\left(g_{i\overline{i}}-\hat g_{i\overline{i}}\right)
-2g^{\overline{i}i}[V_1,e_i]V_1\overline{e}_i(\varphi)-2g^{\overline{i}i}[V_1,\overline{e}_i]V_1e_i(\varphi)+G\nonumber\\
=&g^{\overline{i}i}V_1V_1(g_{i\overline{i}})
-2g^{\overline{i}i}[V_1,e_i]V_1\overline{e}_i(\varphi)-2g^{\overline{i}i}[V_1,\overline{e}_i]V_1e_i(\varphi)+G,\nonumber
\end{align}
where we use the fact that $\hat\omega_t$ vary smoothly on $M\times[0,T_0]$ and hence can be estimated uniformly.

From \eqref{apcma}, we get
\begin{align}
\label{1apcma}
g^{\overline{i}i}\left(\nabla_{V_{1}}V_{1}\right)(g_{i\overline{i}})
=\left(\nabla_{V_{1}}V_{1}\right)(\dot\varphi)
+\left(\nabla_{V_{1}}V_{1}\right)(\log\Omega).
\end{align}
Applying $V_1$ twice to \eqref{apcma} implies
\begin{align}
\label{2apcma}
g^{\overline{i}i}V_{1}V_1(g_{i\overline{i}})
=g^{\overline{p}p}g^{\overline{q}q}\left|V_1(g_{p\overline{q}})\right|^2+V_1V_1(\dot\varphi)+V_1V_1(\log\Omega).
\end{align}
Combining \eqref{tromegaomega0xiajie}, \eqref{eieiuv1v11}, \eqref{eieiuv1v12}, \eqref{giieieiv1v1}, \eqref{1apcma} and \eqref{2apcma} yields
\begin{align}
\label{Llambda2}
&g^{\overline{i}i}e_i\overline{e}_i(\varphi_{V_1V_1})
-g^{\overline{i}i}[e_i,\overline{e}_i]^{(0,1)}(\varphi_{V_1V_1})\\
=&V_1V_1(\dot\varphi)-\left(\nabla_{V_{1}}V_{1}\right)(\dot\varphi)
+g^{\overline{p}p}g^{\overline{q}q}\left|V_1(g_{p\overline{q}})\right|^2\nonumber\\
&-2g^{\overline{i}i}[V_1,e_i]V_1\overline{e}_i(\varphi)-2g^{\overline{i}i}[V_1,\overline{e}_i]V_1e_i(\varphi)+G\nonumber\\
=&\dot{\varphi}_{V_1V_1}
+g^{\overline{p}p}g^{\overline{q}q}\left|V_1(g_{p\overline{q}})\right|^2
-2g^{\overline{i}i}[V_1,e_i]V_1\overline{e}_i(\varphi)-2g^{\overline{i}i}[V_1,\overline{e}_i]V_1e_i(\varphi)+G.\nonumber
\end{align}
Thanks to \eqref{Llambda1} and \eqref{Llambda2}, we get \eqref{lem1formula}.
\end{proof}
\begin{lem}
\label{lemzuihouyige}
At $(x_0,t_0)$, for any $\varepsilon\in(0,1/2]$, there holds
\begin{align}
\label{Ltildeh1}
0\geq&(2-\varepsilon)\sum\limits_{\alpha>1}g^{\overline{i}i}\frac{|e_i(\varphi_{V_{\alpha}V_{1}})|^2}{\lambda_1\left(\lambda_1-\lambda_{\alpha}\right)}
+\frac{g^{\overline{p}p}g^{\overline{q}q}\left|V_1(g_{p\overline{q}})\right|^2}{\lambda_1}\\
&-(1+\varepsilon)\frac{g^{\overline{i}i}|e_i\left(\varphi_{V_{1}V_1}\right)|^2}{\lambda_1^2}
+\frac{\phi'}{2}\sum\limits_pg^{\overline{i}i}(|e_ie_p(\varphi)|^2+|e_i\overline{e}_p(\varphi)|^2)
-\frac{C}{\varepsilon}\sum\limits_{i}g^{\overline{i}i}\nonumber\\
&+\phi''g^{\overline{i}i}\left|e_i\left(|\partial \varphi|_{g_0}^2\right)\right|^2
-De^{-D(\varphi-\varphi_0)}\left(\Delta_{\omega}^C-\ppt\right)\varphi+D^2e^{-D(\varphi-\varphi_0)} |\partial\varphi|_g^2.\nonumber
\end{align}
\end{lem}
\begin{proof}
At $(x_0,t_0)$, we define
\begin{align*}
[V_1,e_i]=\sum\limits_{\alpha=1}^{2n}\tau_{i\alpha}V_{\alpha}
\end{align*}
for some $\tau_{i\alpha}\in\mathbb{C}$ uniformly bounded. This yields
\begin{align*}
\left|[V_1,e_i]V_1\overline{e}_i(\varphi)
+[V_1,\overline{e}_i]V_1\overline{e}_i(\varphi)\right|\leq C\sum\limits_{\alpha=1}^{2n}\left|V_{\alpha}V_1e_i(\varphi)\right|.
\end{align*}
Note that
\begin{align*}
V_{\alpha}V_1e_i(\varphi)
=&V_{\alpha}e_i V_1(\varphi)+V_{\alpha}\left[V_1,e_i\right](\varphi)\\
=&e_i V_{\alpha} V_1(\varphi)+\left[V_{\alpha},e_i\right]V_1(\varphi)+V_{\alpha}\left[V_1,e_i\right](\varphi)\\
=&e_i\left(\varphi_{V_{\alpha}V_1}\right)+e_i(\nabla_{V_{\alpha}}V_1)(\varphi)
+\left[V_{\alpha},e_i\right]V_1(\varphi)+V_{\alpha}\left[V_1,e_i\right](\varphi).
\end{align*}
Therefore, it follows that, at $(x_0,t_0)$,
\begin{align}
\label{fiivieieibaru}
&2\frac{g^{\overline{i}i}\left([V_1,e_i]V_1\overline{e}_i(\varphi)
+[V_1,\overline{e}_i]V_1e_i(\varphi)\right)}{\lambda_1}\\
\leq&C\frac{g^{\overline{i}i}|e_i\left(\varphi_{V_{1}V_1}\right)|}{\lambda_1}
+C\sum\limits_{\alpha>1}\frac{g^{\overline{i}i}|e_i\left(\varphi_{V_{\alpha}V_1}\right)|}{\lambda_1}
+C\sum\limits_{i}g^{\overline{i}i}\nonumber\\
\leq&\varepsilon\frac{g^{\overline{i}i}|e_i\left(\varphi_{V_{1}V_1}\right)|^2}{\lambda_1^2}
+\varepsilon\sum\limits_{\alpha>1}\frac{g^{\overline{i}i}|e_i\left(\varphi_{V_{\alpha}V_1}\right)|^2}{\lambda_1(\lambda_{1}-\lambda_{\alpha})}
+\frac{C}{\varepsilon}\sum\limits_{i}g^{\overline{i}i}
+\frac{C}{\varepsilon}\sum\limits_{i}g^{\overline{i}i}\sum\limits_{\alpha>1}\frac{\lambda_1-\lambda_{\alpha}}{\lambda_1}\nonumber\\
\leq&\varepsilon\frac{g^{\overline{i}i}|e_i\left(\varphi_{V_{1}V_1}\right)|^2}{\lambda_1^2}
+\varepsilon\sum\limits_{\alpha>1}\frac{g^{\overline{i}i}|e_i\left(\varphi_{V_{\alpha}V_1}\right)|^2}{\lambda_1(\lambda_{1}-\lambda_{\alpha})}
+\frac{C}{\varepsilon}\sum\limits_{i}g^{\overline{i}i},\nonumber
\end{align}
where we use the Cauchy-Schwarz inequality for $\varepsilon\in (0,1/2]$ and \eqref{lapg0varphi} to get, at $(x_0,t_0)$,
\begin{align*}
\sum\limits_{\alpha>1}\frac{\lambda_1-\lambda_{\alpha}}{\lambda_1}
=&(2n-1)-\sum\limits_{\alpha>1}\frac{\lambda_{\alpha}(\nabla^2\varphi)-1}{\lambda_1}\\
=&2n+\frac{2n-1}{\lambda_1}-\frac{\Delta u}{\lambda_1}\leq 4n-1+C/\lambda_1\leq C,
\end{align*}
by the assumption that $\lambda_1(\nabla^2\varphi)\gg K\geq 1$.

At $(x_0,t_0)$, since we have
\begin{align*}
0\geq&\left(\Delta_{\omega}^C-\ppt\right)\tilde H\\
=&\frac{1}{\lambda_1}\left(\Delta_{\omega}^C-\ppt\right)\lambda_1-\frac{g^{\overline{i}i}|e_i(\lambda_1)|^2}{\lambda_1^2}
+\phi'\left(\Delta_{\omega}^C-\ppt\right)|\partial\varphi|_{g_0}^2
+\phi''g^{\overline{i}i}\left|e_i(|\partial\varphi|_{g_0}^2)\right|^2\nonumber\\
&-De^{-D\varphi}\left(\Delta_{\omega}^C-\ppt\right)\varphi+D^2e^{-D\varphi}|\partial\varphi|_{g}^2,\nonumber
\end{align*}
and a direct computation yields $e_{i}(\lambda_1)=e_i(\varphi_{V_1V_1})$,
the inequality \eqref{Ltildeh1} follows from \eqref{lemformulaevolutionpartialvarphi}, \eqref{lem1formula} and \eqref{fiivieieibaru}.
\end{proof}
In what follows, we use $C_D$ to denote the constant depending on the initial data and $D$, which is a uniform constant when $D$ is determined.
We split up into different cases.

\textbf{Case 1:} Assume that
\begin{align}
\label{case1condition}
g_{1\overline{1}}(x_0,t_0)\leq D^3 e^{-2D(\varphi(x_0,t_0)-\varphi_0)}g_{n\overline{n}}(x_0,t_0).
\end{align}
Since $\tilde H$ attains maximum at $(x_0,t_0)$, we have $e_{i}(\tilde H)=0$, i.e.,
\begin{align}
\label{eiH}
\frac{e_i(\varphi_{V_1V_1})}{\lambda_1}=De^{-D(\varphi-\varphi_0)}e_i(\varphi)-\phi'e_i\left(|\partial\varphi|_{g_0}^2\right).
\end{align}
Taking $\varepsilon=1/2$, combining \eqref{eiH} and the Cauchy-Schwarz inequality yields
\begin{align}
\label{case1}
&-\frac{3}{2}\frac{g^{\overline{i}i}|e_i(\varphi_{V_1V_1})|^2}{\lambda_1^2}\\
=&-\frac{3}{2}g^{\overline{i}i}\left|De^{-D(\varphi-\varphi_0)}e_i(\varphi)-\phi'e_i\left(|\partial\varphi|_{g_0}^2\right)\right|^2\nonumber\\
\geq&-6\left(\sup\limits_{M\times[0,\tmax)}|\partial\varphi|_{g_0}^2\right)D^2e^{-2D(\varphi-\varphi_0)}\sum\limits_{i}g^{\overline{i}i}
-2(\phi')^2g^{\overline{i}i}\left|e_i\left(|\partial\varphi|_{g_0}^2\right)\right|^2\nonumber\\
=&-6\left(\sup\limits_{M\times[0,\tmax)}|\partial\varphi|_{g_0}^2\right)D^2e^{-2D(\varphi-\varphi_0)}\sum\limits_{i}g^{\overline{i}i}
-\phi''g^{\overline{i}i}\left|e_i\left(|\partial\varphi|_{g_0}^2\right)\right|^2\nonumber
\end{align}
From \eqref{Ltildeh1} and \eqref{case1}, it follows that
\begin{align}
\label{Ltildeh11}
0\geq&
-6\left(\sup\limits_{M\times[0,\tmax)}|\partial\varphi|_{g_0}^2\right)D^2e^{-2D(\varphi-\varphi_0)}\sum\limits_{i}g^{\overline{i}i}
+\frac{\phi'}{2}\sum\limits_pg^{\overline{i}i}\left(|e_ie_p(\varphi)|^2+|e_i\overline{e}_p(\varphi)|^2\right)\\
&-De^{-D(\varphi-\varphi_0)}\left(\Delta_{\omega}^C-\ppt\right)\varphi -C\sum\limits_{i}g^{\overline{i}i},\nonumber
\end{align}
where we discard some positive terms. Thanks to \eqref{apcma} and \eqref{case1condition}, we can deduce
\begin{align*}
C_D\geq g_{1\overline{1}}(x_0,t_0)\geq\cdots\geq  g_{n\overline{n}}(x_0,t_0)\geq C_D^{-1}.
\end{align*}
This, together with Lemma \ref{lemc0} and \eqref{Ltildeh11}, yields
\begin{align}
\label{hessvarphi}
\sum\limits_{i,p}\left(|e_ie_p(\varphi)|^2+|e_i\overline{e}_p(\varphi)|^2\right)(x_0,t_0)\leq C_D,
\end{align}
where we also use
\begin{align}
\label{lapomegavarphiuse}
-De^{-D(\varphi-\varphi_0)}\Delta_{\omega}^C\varphi
=&-De^{-D(\varphi-\varphi_0)}g^{\overline{i}i}\left(g_{i\overline{i}}-\hat g_{i\overline{i}}\right)\\
\geq& -nDe^{-D(\varphi-\varphi_0)}+C_0^{-1}De^{-D(\varphi-\varphi_0)}\sum\limits_{i}g^{\overline{i}i}.\nonumber
\end{align}
From Theorem \ref{thm1order} and \eqref{hessvarphi}, it follows that $\lambda_1(x_0,t_0)$ and hence $\tilde H$ is bounded from above. This completes the proof of Case 1.

\textbf{Case 2:} At $(x_0,t_0)$, assume that
\begin{align}
\label{case2condition}
\frac{\phi'}{4}\sum\limits_pg^{\overline{i}i}\left(|e_ie_p(\varphi)|^2+|e_i\overline{e}_p(\varphi)|^2\right)
\geq 6\left(\sup\limits_{M\times[0,\tmax)}|\partial\varphi|_{g_0}^2\right)D^2e^{-2D(\varphi-\varphi_0)}\sum\limits_{i}g^{\overline{i}i}.
\end{align}
By the same argument as in Case 1, we also have \eqref{Ltildeh11}. From Lemma \ref{lemc0}, \eqref{Ltildeh11}, \eqref{lapomegavarphiuse} and \eqref{case2condition}, it follows that, at $(x_0,t_0)$,
\begin{align}
\label{case2use}
0\geq C^{-1}\sum\limits_p g^{\overline{i}i}\left(|e_ie_p(\varphi)|^2+|e_i\overline{e}_p(\varphi)|^2\right)+(C_0^{-1}D-C)\sum\limits_ig^{\overline{i}i}-C_D.
\end{align}
We choose $D$ sufficiently large such that $C_0^{-1}D-C>1$. Then from \eqref{apcma}, we can deduce the lower and upper bounds of $g_{i\overline{i}}$ and applying the same argument as in Case 1 to \eqref{case2use} completes the proof of Case 2.

\textbf{Case 3:} At $(x_0,t_0)$, neither Case 1 nor Case 2 holds.

In this case, we need to estimate
\begin{align*}
(2-\varepsilon)\sum\limits_{\alpha>1}g^{\overline{i}i}\frac{|e_i(\varphi_{V_{\alpha}V_{1}})|^2}{\lambda_1\left(\lambda_1-\lambda_{\alpha}\right)}
+\frac{g^{\overline{p}p}g^{\overline{q}q}\left|V_1(g_{p\overline{q}})\right|^2}{\lambda_1}
-(1+\varepsilon)\frac{g^{\overline{i}i}|e_i\left(\varphi_{V_{1}V_1}\right)|^2}{\lambda_1^2}
\end{align*}
in Lemma \ref{lemzuihouyige}. For this aim, we define
$$
I:=\Big\{i:\;g_{i\overline{i}}(x_0,t_0)> D^3 e^{-2D(\varphi(x_0,t_0)-\varphi_0)}g_{n\overline{n}}(x_0,t_0)\Big\}.
$$
Since Case 1 does not hold, we have $1\in I$. Without loss of generality, we denote $I=\Big\{1,\cdots,j\Big\}$.
By the similar argument of \cite[Lemma 5.5]{ctw}, we get
\begin{lem}
\label{lemyinyong}
At $(x_0,t_0)$, for any $(0,1/2]$, we have
\begin{align*}
-(1+\varepsilon)\sum\limits_{i\in I}\frac{g^{\overline{i}i}|e_i\left(\varphi_{V_{1}V_1}\right)|^2}{\lambda_1^2}
\geq -\sum\limits_{i}g^{\overline{i}i}-\phi''\sum\limits_{i\in I}g^{\overline{i}i}\left|e_i\left(|\partial\varphi|_{g_0}^2\right)\right|^2
\end{align*}
\end{lem}
Assume that $\lambda_1(x_0,t_0)\geq C_D/\varepsilon^3$, where $D$ and $\varepsilon$ will be chosen uniformly later. The similar arguments in \cite[Lemma 5.6, Lemma 5.7, Lemma 5.8]{ctw} yield
\begin{lem}
\label{lemyinyong2}
At $(x_0,t_0)$, for any $(0,1/6]$, we have
\begin{align*}
&(2-\varepsilon)\sum\limits_{\alpha>1}g^{\overline{i}i}\frac{|e_i(\varphi_{V_{\alpha}V_{1}})|^2}{\lambda_1\left(\lambda_1-\lambda_{\alpha}\right)}
+\frac{g^{\overline{p}p}g^{\overline{q}q}\left|V_1(g_{p\overline{q}})\right|^2}{\lambda_1}
-(1+\varepsilon)\sum\limits_{i\not\in I}\frac{g^{\overline{i}i}|e_i\left(\varphi_{V_{1}V_1}\right)|^2}{\lambda_1^2}\\
\geq& -3\varepsilon \sum\limits_{i\not\in I}\frac{g^{\overline{i}i}|e_i\left(\varphi_{V_{1}V_1}\right)|^2}{\lambda_1^2}-\frac{C}{\varepsilon}\sum\limits_{i}g^{\overline{i}i}.
\end{align*}
\end{lem}
Since $\partial\tilde H(x_0,t_0)=0$, for any $\varepsilon\in(0,1/6]$, the Cauchy-Schwarz inequality yields
\begin{align}
\label{2orderguji}
&-3\varepsilon \sum\limits_{i\not\in I}\frac{g^{\overline{i}i}|e_i\left(\varphi_{V_{1}V_1}\right)|^2}{\lambda_1^2}\\
=&-3\varepsilon \sum\limits_{i\not\in I}g^{\overline{i}i}\left|De^{-D(\varphi-\varphi_0)}e_i(\varphi)-\phi'e_i\left(|\partial\varphi|_{g_0}^2\right)\right|^2\nonumber\\
\geq&-6\varepsilon D^2e^{-2D(\varphi-\varphi_0)}|\partial\varphi|_{g}^2
-6\varepsilon (\phi')^2\sum\limits_{i\not\in I}g^{\overline{i}i}\left|e_i\left(|\partial\varphi|_{g_0}^2\right)\right|^2\nonumber\\
\geq&-6\varepsilon D^2e^{-2D(\varphi-\varphi_0)}|\partial\varphi|_{g}^2
-\phi''\sum\limits_{i\not\in I}g^{\overline{i}i}\left|e_i\left(|\partial\varphi|_{g_0}^2\right)\right|^2.\nonumber
\end{align}
From \eqref{Ltildeh1}, \eqref{lapomegavarphiuse}, \eqref{2orderguji}, Lemma \ref{lemc0}, Lemma \ref{lemyinyong} and Lemma \ref{lemyinyong2}, it follows that, at $(x_0,t_0)$,
\begin{align}
\label{2jiezuihouyigeguji}
0\geq&\frac{\phi'}{2}\sum\limits_pg^{\overline{i}i}(|e_ie_p(\varphi)|^2+|e_i\overline{e}_p(\varphi)|^2)
+\left(C_0^{-1}De^{-D(\varphi-\varphi_0)}-\frac{C_2}{\varepsilon}\right)\sum\limits_{i}g^{\overline{i}i}\\
&-C_2De^{-D(\varphi-\varphi_0)}
+\left(D^2e^{-D(\varphi-\varphi_0)}-6\varepsilon D^2e^{-2D(\varphi-\varphi_0)}\right) |\partial\varphi|_g^2.\nonumber
\end{align}
Choose
$$
D=C_0(6C_2+1),\quad \varepsilon=e^{D\left(\varphi(x_0,t_0)-\varphi_0\right)}/6\in(0,1/6].
$$
Now at $(x_0,t_0)$, by Theorem \ref{thm1order} and \eqref{2jiezuihouyigeguji}, we can deduce
$$
\sum\limits_p g^{\overline{i}i}\left(|e_ie_p(\varphi)|^2+|e_i\overline{e}_p(\varphi)|^2\right)+\sum\limits_i g^{\overline{i}i}\leq C,
$$
for a uniform constant $C>0$.  The similar discussion in Case 1 yields the upper bound for $\lambda_1(x_0,t_0)$ and hence for $\tilde H$, which completes the proof of Case 3.
\end{proof}
\section{Proof of the maximum time existence theorem}
\label{sec6}
In this section, using the priori estimates established earlier, we get the higher order priori estimates and complete the proof of Theorem \ref{mainthm}.
\begin{proof}[Proof of Theorem \ref{mainthm}]
By Theorem \ref{thm2order}, there exists a uniform constant $C>0$ such that
$$
C^{-1}\omega_0\leq\omega\leq C\omega_0.
$$
From the higher order estimates of \cite{chu1607} using the $C^{2,\alpha}$ estimate of \cite{chujianchuncvpde} (see also \cite{TWWY}), it follows that for each $k=0,1,2,\cdots$, there exists uniform constants $C_k>0$ such that
$$
\|\varphi(t)\|_{C^k(\omega_0)}\leq C_k.
$$
These uniform estimates on $[0,\tmax)$ imply that we can take limits of $\varphi(t)$ and get a solution on $[0,\tmax]$. Applying the standard parabolic short time existence theory we get a solution a little beyond $\tmax$, a contradiction. Hence there exists a solution $\varphi$ to \eqref{apcma} on $[0,T_0)$. Taking $\ddbar$ on both sides of \eqref{apcma}, we get a solution to \eqref{acrf} on $[0,T_0)$. Since $T_0<T$ is arbitrary, we get a solution to \eqref{acrf} on $[0,T)$. Uniqueness follows from the uniqueness of the solution to \eqref{apcma}. As mentioned before, the flow \eqref{acrf} cannot extend beyond $T$.
\end{proof}
\section{Some convergence results}
\label{sec7}
In this section, we consider some convergence results of the flow \eqref{acrf}. Firstly, let us recall the case when there exists a volume form $\Omega$ with $\cric(\Omega)=0$ considered by Chu \cite{chu1607}.  The similar argument as in Lemma \ref{dengjiadingyi}, yields that
\begin{lem}[Chu \cite{chu1607}; 2016]
Assume that there exists a volume form $\Omega$ with $\cric(\Omega)=0$. A smooth family $\omega(t)$ of almost Hermitian metrics on $[0,\infty)$ solves the flow \eqref{acrf} if and only if there is a family of smooth functions $\varphi(t)$ for $t\in [0,\infty)$ solves
\begin{align*}
\ppt\varphi=\log\frac{\left(\omega_0+\ddbar\varphi\right)^n}{\Omega},\quad \omega_0+\ddbar\varphi>0,\quad \varphi(0)=0,
\end{align*}
with $\omega(t)=\omega_0+\ddbar\varphi$.
\end{lem}
Chu \cite{chu1607} shows that $\varphi(t)$ converge to $\varphi_{\infty}$ smoothly as $t\longrightarrow \infty$ and $\omega_{\infty}=\omega_0+\ddbar\varphi_{\infty}>0$ with $\cric(\omega_{\infty})=0$.

If the complex structure $J$ is integrable, then this assumption is equivalent to the fact that $c_1^{\mathrm{BC}}(M)=0$, and the convergence result belongs to Gill \cite{gill} who used the crucial zero order estimate from \cite{TW2}.

Secondly, we assume that there exists a volume form $\Omega$ with $\cric(\Omega)<0$, which is equivalent to the fact that $M$ is K\"ahler manifold with negative first Chern class if $J$ is integrable. By Theorem \ref{mainthm}, the flow \eqref{acrf} exists for all the time and we denote by $\tilde \omega(s)$ its solution.

Suppose that $t=\log(s+1)$ and $\omega=\frac{\tilde\omega}{s+1}$. We get a new metric which solves
\begin{align}
\label{nacrf}
\ppt \omega=-\cric(\omega)-\omega,\quad \omega(0)=\omega_0,
\end{align}
for $t\in[0,\infty)$. We claim that the flow \eqref{nacrf} is equivalent to the parabolic Monge-Amp\`ere equation
\begin{align}
\label{napcma}
\ppt\varphi=\log\frac{(\hat\omega+\ddbar\varphi)^n}{\Omega}-\varphi,\quad \hat\omega+\ddbar\varphi>0,\quad \varphi|_{t=0}=0,
\end{align}
where $\hat\omega=-\cric(\Omega)+e^{-t}\left(\cric(\Omega)+\omega_0\right)$ with
\begin{align}
\label{evolutionhatomega}
\ppt\hat\omega=-\cric(\Omega)-\hat\omega,\quad\hat\omega|_{t=0}=\omega_0.
\end{align}
Indeed, assume that $\omega$ is the solution to \eqref{nacrf}. By \eqref{nacrf} and \eqref{evolutionhatomega}, we know that
\begin{align}
\label{gai3}
\ppt(\omega-\hat\omega)=-(\omega-\hat\omega)+\ddbar\log\frac{\omega^n}{\Omega}.
\end{align}
We define $\varphi(t)$ by
\begin{align*}
\varphi(t):=e^{-t}\int_0^te^s\log\frac{\omega(s)^n}{\Omega}\md s,
\end{align*}
for any $t\in [0,\infty)$. This function $\varphi(t)$ satisfies
\begin{align}
\label{shoulianyong}
\ppt\varphi=\log\frac{\omega^n}{\Omega}-\varphi,\quad \varphi(0)=0.
\end{align}
Thanks to \eqref{gai3} and \eqref{shoulianyong}, we can deduce that
$$
\ppt\left(e^t(\omega-\hat\omega-\ddbar\varphi)\right)=0,\quad (\omega-\hat\omega-\ddbar\varphi)|_{t=0}=0,
$$
which implies $\omega=\hat\omega+\ddbar\varphi>0$ for all $t\geq 0$.

Conversely, assume that $\varphi$ is the solution to \eqref{napcma}. Taking $\ddbar$ on the both sides of \eqref{napcma}, we know that
$\omega=\hat\omega+\ddbar\varphi$ is the solution to \eqref{nacrf}.

Note that $\hat\omega$ converge smoothly to $-\cric(\Omega)>0$ as $t\longrightarrow \infty$. Hence there exists a uniform constant $C_0>0$ such that
\begin{align}
\label{cankaoduliangdengjia}
C_0^{-1}\omega_0\leq\hat\omega\leq C_0\omega_0.
\end{align}
It is easy to see that Theorem \ref{thm2} follows  from
\begin{thm}
Let $\varphi(t)$ be the solution to \eqref{napcma} on $M\times[0,\infty)$. Then $\varphi$ converge smoothly to $\varphi_{\infty}$ as $t\longrightarrow \infty$, and
we have
\begin{align}
\cric(\omega_{\infty})=-\omega_{\infty},
\end{align}
where $\omega_{\infty}=-\cric(\Omega)+\ddbar\varphi_{\infty}>0$.
\end{thm}
\begin{proof}
The proof consists of several steps as follows.

\emph{Step 1:} We deduce the uniform estimates for $\varphi$ and $\dot\varphi:=\ppt \varphi$ using the method from \cite{Ca,TZ,Ts} (see also \cite{twjdg}).
A simple maximum principle argument yields that $|\varphi|\leq C$ for a uniform constant $C>0$ independent of $t$.
A direct computation gives
\begin{align*}
\left(\ppt-\Delta_{\omega}^C\right)\dot \varphi
=&-\dot\varphi-\tr_{\omega}\left(\cric(\Omega)+\hat\omega\right)\\
=&-\dot\varphi-n+\Delta_{\omega}^C\varphi-\tr_{\omega}\left(\cric(\Omega)\right),
\end{align*}
i.e.,
\begin{align*}
\left(\ppt-\Delta_{\omega}^C\right)\left(\dot\varphi+\varphi\right)=-n-\tr_{\omega}\left(\cric(\Omega)\right).
\end{align*}
At the minimum $(x_0,t_0)$ of $\varphi+\dot\varphi$, without loss of generality, we can assume $t_0>0$, otherwise we can get the lower bound of $\dot\varphi$ directly. At $(x_0,t_0)$, we have $-\tr_{\omega}\left(\cric(\Omega)\right)\leq n$. This, together with $\cric(\Omega)<0$ and the arithmetic-geometric means inequality, yields $$\dot\varphi+\varphi=\log\frac{\omega^n}{\Omega}
=\log\frac{\omega^n}{\left(-\cric(\Omega)\right)^n}+\log\frac{\left(-\cric(\Omega)\right)^n}{\Omega}
\geq\log\frac{\left(-\cric(\Omega)\right)^n}{\Omega}\geq-C$$
at this point and hence everywhere. Since $|\varphi|\leq C$, we get $\dot\varphi\geq-C$.

Since $e^t\left(\cric(\Omega)+\hat\omega\right)=\cric(\Omega)+\omega_0$, a direct computation yields
$$
\left(\ppt-\Delta_{\omega}^C\right)\left(\dot\varphi+\varphi+nt-e^t\dot\varphi\right)=\tr_{\omega}\omega_0>0.
$$
This, together with the maximum principle, implies that there holds $\dot\varphi\leq Cte^{-t}\leq Ce^{-t/2}$ for $t\geq 1$. The uniform upper bound of $\varphi+\dot\varphi=\log\frac{\omega^n}{\Omega}$ and the arithmetic-geometric means inequality yield that
\begin{align}
\label{guanjianxiajie}
\tr_{\omega}\omega_0\geq n\left(\frac{\omega_0^n}{\omega^n}\right)^{1/n}= n\left(\frac{\Omega}{\omega^n}\frac{\omega_0^n}{\Omega}\right)^{1/n}\geq c,
\end{align}
for a uniform constant $c>0$.

Given \eqref{cankaoduliangdengjia} and \eqref{guanjianxiajie}, we can deduce the first and second order estimates.

\emph{Step 2}: We need the estimate $\sup\limits_{M\times[0,\infty)}|\partial\varphi|_{g_0}\leq C$ for a uniform constant $C>0$.
To see this, we just need to follow the proof of Theorem \ref{thm1order} word for word except replacing \eqref{pptu} by
\begin{align*}
\ppt|\partial\varphi|_{g_0}^2
=&e_k(\dot\varphi)\overline{e}_k(\varphi)+e_k(\varphi)\overline{e}_k(\dot\varphi)\\
=&g^{\overline{i}i}e_k(g_{i\overline{i}})\overline{e}_k(\varphi)
-e_k(\log\Omega)\overline{e}_k(\varphi)
+g^{\overline{i}i}e_k(\varphi)\overline{e}_k(g_{i\overline{i}})
-e_k(\varphi)\overline{e}_k(\log\Omega),\nonumber\\
&-2e_k(\varphi)\overline{e}_k(\varphi).\nonumber
\end{align*}
In the process of calculation of $\left(\Delta_{\omega}^C-\ppt\right)|\partial\varphi|_{g_0}^2$, the term $-2e_k(\varphi)\overline{e}_k(\varphi)$ is harmless. Then all other arguments are the same and hence we get the first order estimate.

\emph{Step 3}: We need the third order estimate $\sup\limits_{M\times[0,\infty)}|\nabla^2\varphi|_{g_0}\leq C$, where $\nabla^2\varphi$ is the Hessian with respect to the Levi-Civita connection of $g_0$. To see this, we still need to follow the proof of Theorem \ref{thm2order} word for word except replacing
\eqref{1apcma} and \eqref{2apcma} by
\begin{align*}
g^{\overline{i}i}\left(\nabla_{V_{1}}V_{1}\right)(g_{i\overline{i}})
=\left(\nabla_{V_{1}}V_{1}\right)(\dot\varphi)
+\left(\nabla_{V_{1}}V_{1}\right)(\log\Omega)+\left(\nabla_{V_{1}}V_{1}\right)(\varphi).
\end{align*}
and
\begin{align*}
g^{\overline{i}i}V_{1}V_1(g_{i\overline{i}})
=g^{\overline{p}p}g^{\overline{q}q}\left|V_1(g_{p\overline{q}})\right|^2+V_1V_1(\dot\varphi)+V_1V_1(\log\Omega)+V_1V_1(\varphi).
\end{align*}
Note that the new term $V_1V_1(\varphi)-\left(\nabla_{V_{1}}V_{1}\right)(\varphi)=\varphi_{V_1V_1}$ is harmless and absorbed by $G$, as required.

\emph{Step 4}: The higher order estimates follows from \cite{chu1607}.

\emph{Step 5}: Convergence result. The second order estimate implies that
$$
C^{-1}\omega_0\leq \omega\leq C\omega_0
$$
for a uniform constant $C>0$. This yields that
\begin{align*}
\left(\ppt-\Delta_{\omega}^C\right)(e^t\dot\varphi)=-\tr_{\omega}\left(\cric(\Omega)+\omega_0\right)\geq-C.
\end{align*}
This, together with the maximum principle, implies that $\dot\varphi\geq -(1+t)e^{-t}\geq -Ce^{-t/2}$.

We can deduce that $\dot\varphi$ converge uniformly to zero exponentially fast, and hence $\varphi$ converge uniformly exponentially fast to a continuous limit function $\varphi_{\infty}$. From the higher order uniform estimates of $\varphi$, it follows that $\varphi_{\infty}$ is smooth and
$$\varphi\longrightarrow \varphi_{\infty}$$
in smooth topology. Therefore, by passing to the limits in \eqref{shoulianyong}, we can deduce that the limit metric $\omega_{\infty}=-\cric(\Omega)+\ddbar\varphi_{\infty}$ satisfies
\begin{align}
\label{sol1}
\log\frac{\omega_{\infty}^n}{\Omega}=\varphi_{\infty}.
\end{align}
Taking $\ddbar$ on the both sides of this equation, we get
$$
\cric(\omega_{\infty})=-\omega_{\infty},
$$
as required.

\emph{Step 6}: We prove that $\omega_{\infty}$ is independent of the initial metric $\omega_0$. Assume that the normalized flow \eqref{nacrf} starts at another almost Hermitian metric $\omega_0'$. The same argument as above implies that there exists a smooth function $\varphi'_{\infty}$ such that
\begin{align}
\label{sol2}
(-\cric(\Omega)+\ddbar\varphi'_{\infty})^n=e^{\varphi_{\infty}'}\Omega,\quad \omega_{\infty}':=-\cric(\Omega)+\ddbar\varphi'_{\infty}>0.
\end{align}
Thanks to \eqref{sol1} and \eqref{sol2}, we get
\begin{align}
\label{sol3}
\omega_{\infty}'^n=(\omega_{\infty}+\ddbar\phi)^n=e^{\phi}\omega_{\infty}^n,
\end{align}
where we denote $\phi:=\varphi_{\infty}'-\varphi_{\infty}$. Applying the maximum principle to \eqref{sol3} shows that $\phi\equiv0$, i.e.,
$$\omega_{\infty}'=\omega_{\infty},$$
as desired.
\end{proof}
\section{An example}
\label{secexample}
In this section, as an example, we study the flow \eqref{acrf} on the (locally) homogeneous manifolds in more detail. We note that the flow \eqref{acrf} can be seen as a $(p,q)$ flow given in \cite{laurent} and hence we can use the method in  \cite{laurent}.

Let $(M,J,\omega_0)$ be a compact almost Hermitian manifold whose universal cover is a Lie group $G$ such that if $\pi:G\longrightarrow M$ is the covering map, then $\pi^{\ast}\omega_0$ and $\pi^{\ast}J$ are left invariant. For example, we can take $M=G/\Gamma$, where $\Gamma$ is a compact discrete subgroup of $G$, including solvmanifolds and nilmanifolds. Then the solution to \eqref{acrf} on $M$ is obtained by pulling down the corresponding flow solution on the Lie group $G$, which stays left invariant.
The flow \eqref{acrf} on $G$ becomes an ordinary differential equation for a family of almost Hermitian metrics $\omega(t)$ with respect to the fixed complex structure $J$ on the Lie algebra $\mathfrak{g}$ of the Lie group $G$.

It is sufficient to consider the flow \eqref{acrf} on Lie group $G$. Since the Chern-Ricci form  $p$ of a left invariant almost Hermitian metric $\omega$ defined by
$$p(X,Y)=-\frac{1}{2}\tr\, J\, \mathrm{ad}[X,Y]+\frac{1}{2}\tr\, \mathrm{ad}\, J[X,Y],\quad \forall\,X,\,Y\in\mathfrak{g}
$$
depends only on $J$ (see for example \cite{P} or \cite[Proposition 4.2]{V2}), the flow \eqref{acrf} becomes
$$
\ppt\omega(t)=-2p^{(1,1)},\quad \omega(0)=\omega_0
$$
with solution $\omega(t)$ given by
$$
\omega_t:=\omega(t)=\omega_0-2tp^{(1,1)}.
$$
We define
$$
p^{(1,1)}=\omega_0(P_0\cdot,\cdot)=\omega_t(P(t)\cdot,\cdot),
$$
which implies that
\begin{align}
P_t:=P(t)=(\mathrm{Id}-2tP_0)^{-1}P_0.
\end{align}
We call $P_0$ (resp. $P_t$) the Chern-Ricci operator of $\omega_0$ (resp. $\omega_t$).
It follows that the maximal existence time $T$ is given by
\begin{equation}
T=\left\{
\begin{array}{ll}
  \infty, &  \quad\text{if}\;P_0\leq 0, \\
  1/(2p_{+}), & \quad\text{otherwise},
\end{array}
\right.
\end{equation}
where $p_{+}$ is the maximal positive eigenvalue of the Chern-Ricci operator $P_0$ of $\omega_0$.

Let $p_1,\cdots,p_{2n}$ of eigenvalues of $P_0$ with the orthonormal basis $e_1,\cdots,e_{2n}$, i.e.,
$$
\omega_0(e_i,\overline{e}_j)=\delta_{ij},\quad P_0(e_{\alpha})=p_{\alpha}e_{\alpha},\quad \alpha=1,\cdots,2n.
$$
A direct calculation yields that the scalar curvature $R(\omega_t)$ of $\omega_t$ is given by
\begin{align}
\label{shuliangqulv}
R(\omega_t)=\tr_{\omega_t}p^{(1,1)}=\tr P_t=\sum\limits_{\alpha=1}^{2n}\frac{p_{\alpha}}{1-2tp_{\alpha}}.
\end{align}
From \eqref{shuliangqulv}, it follows that
\begin{align*}
\frac{\md}{\md t}R(\omega_t)=\sum\limits_{\alpha=1}^{2n}\frac{2p_{\alpha}^2}{(1-2tp_{\alpha})^2}\geq0.
\end{align*}
This implies that $R(\omega_t)$ is strictly increasing unless $P_t\equiv 0$, i.e., $\omega_t\equiv\omega_0$, and that the Chern scalar curvature must blow up in finite  singularities, i.e., if $T<\infty$, then there holds
$$
\int_0^{T}R(\omega_t)\md t=\infty.
$$
Also we note that
\begin{align}
\label{scalar}
R(\omega_t)\leq \frac{C}{T-t},
\end{align}
for a uniform constant $C>0$. We remark that \eqref{scalar} is the claim of \cite[Conjecture 7.7]{SW4} for the K\"ahler-Ricci flow on general compact K\"ahler manifolds (see \cite{gillscalar} for the Chern-Ricci flow on general compact Hermitian manifolds).

%\bibliographystyle{abbrv}
%\bibliography{D:/bib/refs_General}

\begin{thebibliography}{99}
\bibitem{angella} Angella, D. {\em Cohomological aspects in complex non-K\"ahler geometry}, Lecture Notes in Mathematics, 2095. Springer, Cham, 2014.
\bibitem{BS} Brendle, S.; Schoen, R.M. {\em Classification of manifolds with weakly 1/4-pinched curvatures}, Acta Math. {\bf 200} (2008), no. 1, 1--13.

%\bibitem{asgreen}Alesker, S.; Shelukhin, E. \emph{On a uniform estimate for the quaternionic Calabi problem},
%Israel J. Math. \textbf{197} (2013), no.1, 309-327.
\bibitem{Ca} Cao, H.-D. {\em Deformation of K\"ahler metrics to K\"ahler-Einstein metrics on compact K\"ahler manifolds}, Invent. Math. {\bf 81}  (1985), no. 2, 359--372.

\bibitem{csw} Chen, X.; Sun, S.; Wang, B. {\em K\"ahler-Ricci flow, K\"{a}hler-Einstein metric, and $K$-stability}, arXiv: 1508.04397.
\bibitem{CW} Chen, X.; Wang, B. {\em K\"ahler-Ricci flow on Fano manifolds (I)},  J. Eur. Math. Soc. (JEMS) {\bf 14} (2012), no. 6, 2001--2038.
\bibitem{chern}Chern, S.-S. \emph{Characteristic classes of Hermitian manifolds},
Ann. of Math. \textbf{47}(1946), 85-121.

\bibitem{Ch} Cherrier, P. {\em \'Equations de Monge-Amp\`ere sur les vari\'et\'es Hermitiennes compactes}, Bull. Sc. Math (2) {\bf 111} (1987), 343--385.
\bibitem{chujianchuncvpde}Chu, J.  \emph{$C^{2,\alpha}$ regularities and estimates for nonlinear elliptic and parabolic equations in geometry},  Calc. Var. Partial Differential Equations \textbf{55} (2016), no. 1, 1-20.
\bibitem{chu1607}Chu, J.  \emph{The parabolic Monge-Amp\`{e}re equation on compact almost Hermitian manifolds},  arXiv: 1607. 02608.
\bibitem{ctw} Chu, J.; Tosatti, V.; Weinkove, B. \emph{The Monge-Amp\`ere equation for non-integrable almost complex structures},   arXiv:1603.00706.

\bibitem{ehrensmann}Ehresmann, C.;  Libermann, P. \emph{Sur les structures presque hermitiennes isotopes},
C. R. Acad. Sci. paris  \textbf{232}(1951), 1281-1283.
\bibitem{ftwz}
Fang, S.; Tosatti, V.; Weinkove, B.; Zheng, T. \emph{Inoue surfaces and the Chern-Ricci flow}, J. Funct. Anal.  \textbf{271} (2016), no.11, 3162-3185.
\bibitem{FIK} Feldman, M.; Ilmanen, T.; Knopf, D. {\em Rotationally symmetric shrinking and expanding gradient K\"ahler-Ricci solitons},  J. Differential Geometry {\bf 65}  (2003),  no. 2, 169--209.
%\bibitem{gauduchon1} Gauduchon, P. {\em Le th\'eor\`eme de l'excentricit\'e nulle}, C. R. Acad. Sci. Paris S\'er. A-B {\bf 285} (1977), no. 5, A387--A390.
\bibitem{gau}Gauduchon, P. \emph{Hermitian connection and Dirac operators}, Boll. Unione Mat. Ital. B \textbf{11}(1997), 257-288.
%\bibitem{gerhardt}Gerhardt, C.  \emph{Closed Weingarten hypersurfaces in Riemannian manifolds}, J. Differential Geom. \textbf{43} (1996), 612-641.
\bibitem{gill}
Gill, M. \emph{Convergence of the parabolic complex Monge-Amp\`{e}re equation on compact Hermitian manifolds}, Comm. Anal. Geom. \textbf{19} (2011), no. 2, 277-303.
\bibitem{gillmmp}Gill, M. \emph{The Chern-Ricci flow on smooth minimal models of general type}, arXiv: 1307.0066.
\bibitem{gillscalar}Gill, M.; Smith, D. \emph{The behavior of Chern scalar curvature under Chern-Ricci flow},  Proc. Amer. Math. Soc. \textbf{143}(2015), no. 11, 4875-4883.
\bibitem{Ha} Hamilton, R.S. {\em Three-manifolds with positive Ricci curvature}, J. Differential Geom. {\bf 17} (1982), no. 2, 255--306.
\bibitem{Ha2} Hamilton, R.S. {\em The formation of singularities in the Ricci flow}, Surveys in differential geometry, Vol. II (Cambridge, MA, 1993), 7--136, Int. Press, Cambridge, MA, 1995.
\bibitem{kobayashi} Kobayashi, S. \emph{Natural connections in almost complex manifolds}, Contemporary Mathematics \textbf{332} (2003), 153-170.
\bibitem{kn2} Kobayashi, S.; Nomizu, K. \emph{Foundations of differential geometry, Volume II}, interscience (1969), Wiley, New York.
\bibitem{laurent}Lauret, J. \emph{Curvature flows for almost-hermitian Lie groups}, Trans. Amer. Math. Soc. \textbf{367} (2015), 7453-7480.
\bibitem{lr}Lauret, J.; Rodr\'{i}guez-Valencia, E. \emph{On the Chern-Ricci flow and its solitons for Lie groups}, Math. Nachr. 288 (2015), no. 13, 1512-1526.
\bibitem{LY} Liu, K.; Yang, X. {\em Geometry of Hermitian manifolds},  Internat. J. Math. {\bf 23} (2012), no. 6, 1250055, 40 pp.
\bibitem{niexiaolan} Nie, X. \emph{Weak Solutions of the Chern-Ricci flow on compact complex surfaces}, arXiv: 1701. 04965.

\bibitem{P1} Perelman, G. {\em The entropy formula for the Ricci flow and its geometric applications},
preprint, arXiv:math/0211159.
\bibitem{PSSW} Phong, D. H.; Song, J.; Sturm, J.; Weinkove, B. {\em The K\"ahler-Ricci flow and the $\bar\partial$ operator on vector fields}, J. Differential Geom. {\bf 81} (2009), no. 3, 631--647.
\bibitem{PS2} Phong, D. H.; Sturm, J. {\em On stability and the convergence of the K\"ahler-Ricci flow}, J. Differential Geom. {\bf 72} (2006), no. 1, 149--168.
\bibitem{P} Pook, J. {\em Homogeneous and locally homogeneous solutions to symplectic curvature flow}, arXiv: 1202. 1427.
\bibitem{schouten}Schouten, J. A.; van Danzig, D.
\emph{\"{U}ber unit\"{a}re Geometrie}, Math. Ann. \textbf{103} (1930), 319-346.
\bibitem{ST} Song, J.; Tian, G. {\em The K\"ahler-Ricci flow on surfaces of positive Kodaira dimension}, Invent. Math. {\bf 170} (2007), no. 3, 609--653.
\bibitem{ST2} Song, J.; Tian, G. {\em Canonical measures and K\"ahler-Ricci flow}, J. Amer. Math. Soc. \textbf{25} (2012),
 no. 2, 303-353.
\bibitem{ST3} Song, J.; Tian, G. {\em The K\"ahler-Ricci flow through singularities}, Invent. Math. \textbf{207} (2017), no. 2, 519-595.
\bibitem{SW} Song, J.; Weinkove, B. {\em The K\"ahler-Ricci flow on Hirzebruch surfaces}, J. Reine Angew. Math. {\bf 659} (2011), 141--168.
\bibitem{SW2} Song, J.; Weinkove, B. {\em Contracting exceptional divisors by the K\"ahler-Ricci flow},  Duke Math. J. {\bf 162} (2013), no. 2, 367--415.
\bibitem{SW3} Song, J.; Weinkove, B. {\em Contracting exceptional divisors by the K\"ahler-Ricci flow, II}, Proc. London Math. Soc. \textbf{108} (2014), no. 6, 1529-1561.
\bibitem{SW4} Song, J.; Weinkove, B. {\em An introduction to the K\"ahler-Ricci flow}, in {\em An introduction to the K\"ahler-Ricci flow}, 89--188, Lecture Notes in Math., 2086, Springer, Heidelberg, 2013.
\bibitem{SY} Song, J.; Yuan, Y. {\em Metric flips with Calabi ansatz},  Geom. Funct. Anal. {\bf 22} (2012), no. 1, 240--265.
%\bibitem{spruck} Spruck, J. {\em Geometric aspects of the theory of fully nonlinear elliptic equations}, Global theory of minimal surfaces, Amer. Math. Soc., Providence, RI, 2005, 283--309.
\bibitem{StT} Streets, J.; Tian, G. {\em A parabolic flow of pluriclosed metrics}, Int. Math. Res. Not. IMRN {\bf 2010}, no. 16, 3101--3133.
\bibitem{StT2} Streets, J.; Tian, G. {\em Hermitian curvature flow}, J. Eur. Math. Soc. (JEMS) {\bf 13} (2011), no. 3, 601--634.
\bibitem{StT3} Streets, J.; Tian, G. {\em Regularity results for pluriclosed flow},  Geom. Topol. {\bf 17} (2013), no. 4, 2389--2429.
\bibitem{szkrf} Sz\'ekelyhidi, G. {\em The K\"ahler-Ricci flow and K-stability},  Amer. J. Math. {\bf 132} (2010), no. 4, 1077--1090.
\bibitem{sz} Sz\'ekelyhidi, G. {\em Fully non-linear elliptic equations on compact Hermitian manifolds}, arXiv:1501.02762v2.
\bibitem{stw1503}Sz\'ekelyhidi, G.; Tosatti, V.; Weinkove, B.  \emph{Gauduchon metrics   with prescribed volume form}, arXiv:1503.04491.
\bibitem{Ti} Tian, G. {\em New results and problems on K\"ahler-Ricci flow}, Ast\'erisque No. {\bf 322} (2008), 71--92.
\bibitem{TZ} Tian, G.; Zhang, Z. {\em On the K\"ahler-Ricci flow on projective manifolds of general type}, Chinese Ann. Math. Ser. B {\bf 27} (2006), no. 2, 179--192.
\bibitem{To0}
Tosatti, V.
\emph{A general Schwarz lemma for almost-Hermitian manifolds},
Comm. Anal. Geom. {\bf 15} (2007), no. 5, 1063--1086.
\bibitem{To1} Tosatti, V. {\em K\"ahler-Ricci flow on stable Fano manifolds}, J. Reine Angew. Math. {\bf 640} (2010), 67--84.
\bibitem{tosattikawa} Tosatti, V. \emph{KAWA lecture notes on the K\"{a}hler-Ricci flow}, arXiv:1508.04823.
\bibitem{TWWY} Tosatti, V.; Wang, Y.; Weinkove, B.; Yang, X. {\em $C^{2,\alpha}$ estimates for nonlinear elliptic equations in complex and almost complex geometry}, Calc. Var. Partial Differential Equations {\bf 54} (2015), no. 1, 431--453.
\bibitem{TW1} Tosatti, V.; Weinkove, B. {\em Estimates for the complex Monge-Amp\`ere equation on Hermitian and balanced manifolds}, Asian J. Math. {\bf 14} (2010), no.1, 19--40.
\bibitem{TW2} Tosatti, V.; Weinkove, B. {\em The complex Monge-Amp\`ere equation on compact Hermitian manifolds}, J. Amer. Math. Soc. {\bf 23} (2010), no.4, 1187--1195.
\bibitem{twjdg} Tosatti, V.; Weinkove, B. \emph{On the evolution of a Hermitian metric by its Chern-Ricci form}, J. Differential Geom. \textbf{99} (2015), no. 1, 125-163.
\bibitem{twcomplexsurface}
Tosatti, V.; Weinkove, B. \emph{The Chern-Ricci flow on complex surfaces}, Compos. Math. \textbf{149} (2013), no. 12, 2101-2138.

%\bibitem{tw1305} Tosatti, V.; Weinkove, B.  \emph{The Monge-Amp\`ere equation for $(n-1)$-plurisubharmonic functions on a compact K\"ahler manifold}, J. Amer. Math. Soc. \textbf{30} (2017), no. 2, 311-346.
%
%\bibitem{tw1310} Tosatti, V.; Weinkove, B.  \emph{Hermitian metrics, $(n-1, n-1)$ forms and Monge-Amp\`ere equations}, arXiv:1310.6326.


\bibitem{twymathann}
Tosatti, V.; Weinkove, B.; Yang, X. \emph{Collapsing of the Chern-Ricci flow on elliptic surfaces}, Math. Ann. \textbf{362} (2015), no. 3-4, 1223-1271.
\bibitem{twyau}
Tosatti, V.; Weinkove, B.; Yau, S.-T.
\emph{Taming symplectic forms and the Calabi-Yau Equation},
Proc. London Math. Soc. \textbf{97} (2008), no. 3, 401-424.
\bibitem{Ts} Tsuji, H. \emph{Existence and degeneration of K\"ahler-Einstein metrics on minimal algebraic varieties of general type}, Math. Ann. {\bf 281} (1988), no. 1, 123--133.
\bibitem{Ts2} Tsuji, H. {\em Generalized Bergmann metrics and invariance of plurigenera}, preprint, arXiv:math/9604228.
\bibitem{V} Vezzoni, L. {\em On Hermitian curvature flow on almost complex manifolds}, Differential Geom. Appl. 29 (2011), no. 5, 709-722.
\bibitem{V2} Vezzoni, L. {\em A note on canonical Ricci form on $2$-step nilmanifolds}, Proc. Amer. Math. Soc. 141 (2013), 325-333.
\bibitem{weinkovekrf}Weinkove, B. {\em The K\"ahler-Ricci flow on compact K\"ahler manifolds}, Lecture Notes for Park City Mathematics Institute 2014.
\bibitem{yangxiaokui}Yang, X. \emph{The Chern-Ricci flow and holomorphic bisectional curvature},  Sci. China Math. \textbf{59} (2016), no. 11, 2199-2204.
\bibitem{Y} Yau, S.-T. {\em On the Ricci curvature of a compact K\"ahler
manifold and the complex Monge-Amp\`ere equation, I}, Comm. Pure Appl. Math. \textbf{31} (1978), no.3, 339--411.
\bibitem{zhengtaocjm} Zheng, T. \emph{The Chern-Ricci flow on Oeljeklaus-Toma manifolds}, Canad. J. Math. \textbf{69} (2017), no. 1, 220-240.

\end{thebibliography}
\end{CJK}
\end{document}